\documentclass[10pt,twoside]{siamltex}
\usepackage{amsfonts, graphicx}
\usepackage{hyperref}

\def\cvd{~\vbox{\hrule\hbox{%
     \vrule height1.3ex\hskip0.8ex\vrule}\hrule } }

\newtheorem{thm}{Theorem}[section]

\newtheorem{conj}[thm]{Conjecture}

\newtheorem{obs}[thm]{Observation}
\newtheorem{defn}[thm]{Definition}

\def\mvr{{\rm mvr}}

\def\mvr{{\rm mvr}}
\def\mr{{\rm mr}}

\def\be{\begin{equation}}
\def\ee{\end{equation}}
\def\bee{\[}
\def\eee{\]}

\def\bra{\langle}
\def\ket{\rangle}

\newcommand{\reals}{\mathbb{R}}
\newcommand{\complex}{\mathbb{C}}

\begin{document}



\bibliographystyle{plain}

\title{Minimum Vector Rank and Complement Critical Graphs 
\thanks{Received by the editors on Month x, 200x.
Accepted for publication on Month y, 200y   Handling Editor: .}}

\author{
Xiaowei Li\thanks{Department of Mathematics and Computer Science, Saint Mary's College of California, 
Moraga, CA 94556}
\and
Michael Nathanson\footnotemark[2]~\thanks{Corresponding author: man6@stmarys-ca.edu}
\and 
Rachel Phillips\thanks{Graduate Program in Analytics, University of San Francisco, San Francisco, CA 94117} \newline All authors supported by the Saint Mary's School of Science Summer Research Program}


\pagestyle{myheadings}
\markboth{X.\ Li, M.\ Nathanson, and R.\ Phillips}{Minimum Vector Rank and Complement Critical Graphs}
\maketitle

\begin{abstract}
Given a graph $G$, a real orthogonal representation of  $G$ is a function from its set of vertices  to $\reals^d$ such that two vertices are mapped to orthogonal vectors if and only if they are not neighbors. The minimum vector rank of a graph is the smallest dimension $d$ for which such a representation exists. This quantity is closely related to the minimum semidefinite rank of $G$, which has been widely studied. Considering the minimum vector rank as an analogue of the chromatic number, this work defines critical graphs as those for which the removal of any vertex decreases the minimum vector rank; and complement critical graphs as those for which the removal of any vertex decreases the minimum vector rank of either the graph or its complement. It establishes necessary and sufficient conditions for certain classes of graphs to be complement critical, in the process calculating their minimum vector rank. In addition, this work demonstrates that complement critical graphs form a sufficient set to prove the Graph Complement Conjecture, which remains open. 
\end{abstract}



\section{Introduction}

There is an extensive, rich literature connecting ideas in linear algebra with those in graph theory. In particular, the association of a graph with a matrix (or a class of matrices) allows properties of the graph to be encoded with algebraic structure. The minimum rank of a graph is the smallest rank among all Hermitian matrices whose zero pattern away from the diagonal is the same as the zero pattern of the graph's adjacency matrix. The study of such graph parameters was initiated in \cite{Nylen} and pursued during a 2006 American Institute of Mathematics workshop \cite{AIMW}, generating much interest since then (e.g., \cite{ZFS, ZFP, Outerplanar, GCC, OMR, MSR,OR, UM}). See \cite{AS, SII} for thorough surveys of earlier results and lines of inquiry. 

The present work was inspired by the work of Haynes, et al. \cite{OVC}, which looked at orthogonal representations of graphs as an analogue of graph coloring, a strategy which we will employ here. In particular, we will look for the smallest dimension in which we can embed an orthogonal representation of a  graph. One possible interpretation of such a representation is as a set of quantum states, necessarily unit vectors, in which case the associated graph $G$ is the confusability graph of these states \cite{DSW, QCN}. We will not pursue this interpretation here, but it motivates our desire to focus on representations consisting of nonzero vectors. By contrast, much of the literature focuses on the relation between graphs and the rank of their associated matrices (e.g., \cite{ZFP,OMR,SII}), which does not explicitly give a ``coloring'' of a graph using vectors. The ideas are closely related, however, and the differences are mostly linguistic.

The main theme of our current work is complement criticality. Recall that a critical graph is one such that any proper subgraph has a strictly smaller chromatic number. We wish to look at analogous graphs with respect to the minimal vector rank, defining vector-critical graphs to be those such that any proper {\em induced} subgraph has a strictly smaller minimum vector rank. In addition, it will also be interesting to look at {\em complement critical graphs}, in which for any proper induced subgraph, either the subgraph or its complement has a strictly smaller minimum vector rank.  That is, if $H$ is a proper induced subgraph of $G$, then necessarily  $\mvr(H)+\mvr(\overline{H}) < \mvr(G)+\mvr(\overline{G})$.

On question posed during the 2006 AIM workshop has gained much attention: How large can the sum of the minimum rank of a graph and its complement be? This question (and its conjectured answer) has become known as the Graph Complement Conjecture. We will state this conjecture in Section \ref{MVR Overview} and show its connections to our current work. 

The rest of the paper is organized as follows. Section \ref{MVR Overview} introduces basic definitions and properties of the minimum vector rank. In Section \ref{Results}, we give a precise definition of complement critical graphs and explore its immediate consequences. Section \ref{Complements} establishes formulas for the minimum vector rank of the complement of certain sparse graphs, which allows us in Section \ref{Families} to give necessary and sufficient conditions for certain types of graphs to be complement critical. Finally, in Section \ref{MRplus} we briefly restate our results in terms of the minimum semidefinite rank, to illuminate the connection to other work. There is also an Appendix, in which we include some of the more technical proofs from Section \ref{Complements}.


\section{The Minimum Vector Rank of a Graph}\label{MVR Overview}
We recall some standard definitions: A {\em graph} is a pair $G=(V,E)$, where $V$ a finite non-empty set of vertices and $E \subset V \times V$ is the set of edges. We will restrict our focus to {\em simple}  graphs, in which edges are undirected and no loops or multiple edges are allowed. The {\em order} of a graph $G$, denoted $|G|$, is the size of its vertex set. Two vertices $u$ and $v$ in $V$ are {\em neighbors} if and only if $uv \in E$; the neighborhood of $v$ is denoted $N(v) := \{u \in V: uv \in E \}$. The {\em degree} of a vertex is its number of neighbors: $\deg(v) = |N(v)|$. A {\em pendant} vertex has exactly one neighbor, while an {\em isolated} vertex has no neighbors. Two vertices $u$ and $v$ are  {\em duplicate} vertices if they are neighbors and if they share all other neighbors. That is, $N(u) \cup \{u\}= N(v) \cup \{v\}$. 

Building on \cite{AIMW} and works which have followed, we consider the {\em minimum rank} of a graph $G$ on $n$ vertices (denoted $\mr(G)$) as the minimum rank among $n \times n$ Hermitian matrices $M$ such that for all $i \ne j$, $M_{ij} \ne 0 $ if and only if vertices $i$ and $j$ are connected by an edge. (This is equivalent to the fact that the off-diagonal entries of $M$ have the same pattern of zeroes as the adjacency matrix of $G$.) The  {\em minimum semidefinite rank} $\mr_+(G)$ is similarly defined with the additional requirement that $M$ be a positive semidefinite matrix; this is also sometimes indicated by the alternate notation $\mbox{msr}(G)$ \cite{UM}. 

In this work, we wish to assign a nonzero vector to every vertex. This is a well-studied idea with, unfortunately, inconsistent notation and terminology. We will use the following definition: 
\begin{defn}\rm Let $G = (V,E)$ be a simple graph. An {\em orthogonal vector representation} of $G$ is a function 
\bee
\phi: V \rightarrow \reals^d \mbox{ such that } \left\{ \begin{array}{lll}
\bra \phi(v_i), \phi(v_i)\ket &>& 0  \mbox{ for all } v_i \in V \cr
\bra \phi(v_i), \phi(v_j)\ket &= &0 \mbox{ if } i \ne j \mbox{ and } v_iv_j \notin E \cr
\bra \phi(v_i), \phi(v_j)\ket   &\ne & 0 \mbox{ otherwise}\end{array}\right.
\eee
\end{defn}
We interpret this as an analogue of vertex coloring in the sense of assigning a vector to each vertex with conditions on adjacent vertices. This is actually a stronger condition than coloring, as it is restricted both by a vertex's neighbors and by its non-neighbors. Note also that our orthogonal representation is defined as mapping to a {\em real} vector space; and, in fact, there exist graphs for which there exists an orthogonal vector representation in $\complex^3$ but not in $\reals^3$.\cite{ZFP} However, all of the results in this paper hold whether the underlying field is $\reals$ or $\complex$. 

We note that in other settings, Orthogonal Vector Representations allow some vertices to be assigned the zero vector. We explicitly disallow this here; this means that we can always renormalize any representation to make all vectors unit vectors. These vector representations give us our graph parameter of primary interest for this paper, in which we follow the definition given in \cite{UM}: 
\begin{defn}\rm
The {\em minimum vector rank} of a graph $G$ is the smallest integer $d$ such that there exists an orthogonal vector representation of $G$ in $\reals^d$. We denote this as $\mvr(G)$. 
\end{defn}

The minimum vector rank is closely related to the minimum semidefinite rank described above. To make this connection with other parameters explicit, we offer the following observations:
\begin{obs}
For a simple graph $G$, 
\begin{enumerate}
\item{} $\mvr(G) \ge mr_+(G)$, and $\mvr(G) = \mr_+(G)$ unless $G$ contains isolated vertices. 

In fact, if $G$ contains $r$ isolated vertices, then $r = \mvr(G) - \mr_+(G)$. \cite{UM}
\item{} $\mvr(G) = \rm{jmr}_+(G):= \mr_+(G \vee K_1)$, as  in \cite{AS}. 
\item{} $\mvr(G) \ge \chi_v(\overline{G})$, defined in \cite{OVC}. 
\end{enumerate}
\end{obs}

We can easily adapt known properties of the minimum semidefinite rank (from \cite{AS}, \cite{OVC}, et al.)  to give properties of the minimum vector rank. Recall that if two graphs $G_1 = (V_1, E_1)$ and $G_2 = (V_2, E_2)$ have disjoint sets of vertices, then their {\em union} $G_1 \cup G_2$ has vertices $V = V_1 \cup V_2$ and edges $E = E_1 \cup E_2$. If $G_1$ and $G_2$ are connected graphs, then they are the connected components of $G_1 \cup G_2$. Similarly, $G_1 \vee G_2$ is called  the {\em join} of $G_1$ and $G_2$ and contains all of the  edges in $G_1 \cup G_2$ plus additional edges connecting every vertex in $G_1$ to every vertex in $G_2$. We note that for any nonempty graphs $G_1, G_2$, the graph $G_1 \vee G_2$ is connected and $\overline{G_1 \vee G_2} = \overline{G_1} \cup \overline{G_2}$. 

\begin{obs}
Let $G$ be a graph on $n \ge 1$ vertices.
\begin{enumerate}
\item If H an induced subgraph of G then, $\mvr(H)\leq \mvr(G)$. 
\item{} If $G = G_1 \cup G_2$, then $\mvr(G) = \mvr(G_1) + \mvr(G_2)$. 
\item{} If $G = G_1 \vee G_2$, then $\mvr(G) = \max( \mvr(G_1),  \mvr(G_2))$. 
\item In particular, if $K_n$, is the complete graph on $n$ vertices, then $\mvr(K_{n})=1$ and $\mvr(\overline{K_n}) = n$. 
\end{enumerate}
\end{obs}

\begin{proof} (1) is immediate, since any orthogonal representation of $G$ restricts to an orthogonal representation of $H$. (2) also follows immediately from the definition. (3) is equivalent to Theorem 3.2  in \cite{SII}. (4) follows from the fact that $K_n$ is the join of $n$ copies of $K_1$ and $\overline{K_n}$ is the union of $n$ copies of $K_1$. \end{proof}

\section{Complement Critical Graphs}\label{Results}
One can consider the minimum vector rank is as a generalization of ordinary graph coloring in which we assign vectors to the vertices instead of colors. This approach was taken in \cite{OVC}, although they defined their quantities differently. The primary study of this paper is a further extension of the analogy with graph coloring. Recall that a graph $G$ is {\em critical} if $\chi(H) < \chi(G)$ for all subgraphs $H$ of $G$. We propose the following as a natural generalization: 
\begin{defn}\rm 
A graph $G$ is {\em vector~critical} if, for any induced subgraph $H$ of $G$, $\mvr(H) < \mvr(G)$.\end{defn}
We know that  $\mvr(H) \le \mvr(G)$ for any induced subgraph $H$; $G$ is vector critical if this inequality is strict for all $H \ne G$. We will refer to such graphs as {\bf critical} when there is no ambiguity.

The class of critical graphs is fairly restrictive. Since we are looking at induced subgraphs, which are defined in terms of a particular vertex subset, it makes sense to consider a graph and its complement as a pair; the vertices of $H$ determine the structure of $\overline{H}$. This suggests the following definition:
\begin{defn} \rm
A graph $G$ is {\em complement critical} if, for any proper induced subgraph $H$ of $G$ 
\bee \mvr(H) + \mvr({\overline{H}}) < \mvr(G)+\mvr(\overline{G})\eee
\end{defn}
The sum in this inequality is a convenient shorthand for saying that either $\mvr(H) < \mvr(G)$ or $\mvr(\overline{H}) < \mvr(\overline{G})$; removing any vertices from $G$ will decrease at least one of the minimum vector ranks.  Complement criticality is clearly a weaker condition than criticality.  Every vector critical graph is complement critical, but we will see examples in Section \ref{Families} of graphs which are complement critical but not vector critical.

Writing the definition this way also connects it with the larger conversation about the Graph Complement Conjectures (see, e.g., \cite{GCC, ZFS}), which were initially proposed at the 2006 AIM workshop. Despite many positive contributions, the conjecture and several variants remain unproven. We state the two primary conjectures and append a third given in the language of minimum vector rank.

\begin{conj}[Graph Complement Conjectures]\label{GCC}
For any simple graph $G$ on $n$ vertices,
\be
\mr(G)+ \mr(\overline G) \le n+2 \label{mrGCC} \ee\be
\mr_+(G) + \mr_+(\overline G) \le n+2 \label{mrpGCC} \ee\be
\label{mvrGCC}  \mvr(G) + \mvr(\overline G) \le n+2\ee
\end{conj}
Since $\mbox{mr}(G) \le \mbox{mr}_+(G) \le \mvr(G)$ for any graph $G$, it is clear that (\ref{mvrGCC}) implies (\ref{mrpGCC}), which in turn implies (\ref{mrGCC}). What is less obvious is that if (\ref{mrpGCC}) holds for all simple graphs $G$, then so does (\ref{mvrGCC}). This is good news, since we are considering the minimum vector rank as a small tweak of the minimum semidefinite rank. 

The equivalence follows in a straightforward way: Suppose there exists a graph $G$ for which $\mvr(G) + \mvr(\overline G) > n+2$. Without loss of generality, we assume that $G$ is connected (since for any graph at least one of $G$ and $\overline{G}$ must be connected \cite{West}). If we then look at the graph $G' =G \cup K_1$, then $\mr_+(G') = \mr_+(G)+1 = \mvr(G)+1$ since $G$ is connected; and $\mr_+(\overline{G'}) = \mr_+(\overline{G} \vee K_1) = \mvr(\overline{G})$. This implies that 
\bee
\mr_+(G') + \mr_+(\overline{G'})  =  \mvr(G) + 1 + \mvr(\overline{G}) > 1 + n + 2 = |G'| + 2
\eee
So, if there is a counterexample to (\ref{mrpGCC}), then we can build one for (\ref{mvrGCC}). Thus, the validity of (\ref{mrpGCC}) for all graphs is equivalent to the validity of (\ref{mvrGCC}) for all graphs. 

None of these inequalities has been proven to date, but the idea of complement critical graphs gives yet another way to think about this. 
\begin{proposition}\label{equivalence}
If the inequality (\ref{mvrGCC}) holds for all complement critical graphs $G$, then it holds for all simple graphs and Conjecture \ref{GCC} is true.
\end{proposition}

\begin{proof} Suppose $G$ is a graph which is not complement critical. Define ${\cal S}(G)$ to be the set of all proper induced subgraphs $H$ of $G$ such that $\mvr(H) = \mvr(G)$ and $\mvr(\overline{H}) = \mvr(\overline{G})$. Since $G$ is not complement critical, ${\cal S}(G)$ is not empty. Choose a graph $H \in {\cal S}(G)$ with the minimum number of vertices. For any proper induced subgraph $K$ of $H$, $|K| < |H|$ implies that $K \notin {\cal S}(G)$, which in turn implies that $K \notin {\cal S}(H) \subseteq {\cal S}(G)$. This means that $ {\cal S}(H) $ is empty and $H$ is complement critical. 

Note that this argument follows from the fact that the minimum vector rank is monotone on induced subgraphs; an analogous property exists for critical graphs based on chromatic number (mentioned in \cite{West}) or on any similarly defined quantity. It is not inherently about the minimum vector rank.

The proposition follows quickly from this simple fact. Suppose our graph $G$ violates the graph complement conjecture: $\mvr(G) + \mvr(\overline{G}) > |G| + 2$. We consider the complement critical induced subgraph $H$ from above and see that 
\bee
\mvr(H) + \mvr(\overline{H}) = \mvr(G) + \mvr(\overline{G}) > |G| + 2 \ge |H| + 2
\eee
which implies that the graph $H$ also violates the inequality (\ref{mvrGCC}). 

Thus, if (\ref{mvrGCC}) is violated by a graph which is not complement critical, it must also be violated for some complement critical graph. Taking the contrapositive gives us Proposition \ref{equivalence}.
\end{proof}

Proposition \ref{equivalence} says  the complement critical graphs form a sufficient set for the Graph Complement Conjectures. The conjecture need only be verified for this smaller set to imply its validity for all graphs. In the remainder of the paper, we explore classes of complement critical graphs. 


\section{Minimum Vector Rank of Graph Complements}\label{Complements}

In preparation for our discussion of families of complement critical graphs, it will be useful to establish formulas for the minimum rank of the complement of certain types of graphs. This work starts with the bounds given in \cite{OR} on the minimum ranks of several families of graph complements; in particular  that  $\mvr(\overline{C_n}) \le 4$ for any cyclic graph $C_n$ and, as a result, that $\mvr(\overline{U}) \le 4$ for any unicyclic graph $U$. We give precise conditions which dictate the minimum vector rank of these graph complements and then extend them to larger classes of graphs. 

\subsection{Trees and Cycles}

The simplest sparse graphs are trees. The complement of a tree $T$ was shown to satisfy $\mvr(\overline{T}) = \mr_+(\overline{T}) \le 3$ in Theorem 3.16 in \cite{ZFP}. We include the full result for completeness: 
\begin{proposition}\label{mvr tree comp.}
Let $T$ be a tree of order $n\ge 2$, then
\bee
\mvr(\overline{T})=\left\{\begin{array}{ll}
2 & \textrm{if $T=K_{1,n-1}$}\\
3 & \textrm{otherwise}
\end{array}\right.
\eee
\end{proposition}

\begin{proof} Direct calculation shows that  $\mvr(\overline{K_{1,n}})=\mvr(K_1 \cup K_n)=\mvr(K_1)+\mvr(K_n)=2$. On the other hand, if $T$ is not a star, then $P_4$ is an induced subgraph of $T$, implying that $\mvr(\overline{T}) \ge \mvr(\overline{P_4}) =3$. Combining with the inequality $\mvr(\overline{T}) \le 3$ from \cite{ZFP} gives equality. \end{proof}

Our other starting point is a cyclic graph. The authors of \cite{ZFP} give a proof that $\mr(\overline{C_n}) = 3$ for $n \ge 5$. Their proof gives an explicit orthogonal representation, which implies that in fact $\mr_+(\overline{C_n}) = \mvr(\overline{C_n}) = 3$ for $n \ge 5$. A different proof of this fact was given in \cite{MSR}. Direct calculation shows that $\mvr(\overline{C_3}) = \mvr(3K_1) = 3$ and $\mvr(\overline{C_4}) = \mvr(2P_2) = 2$, which gives us a complete characterization of $\mvr(\overline{C_n})$: 

\begin{proposition}\cite{ZFP, MSR}\label{mvr comp. cycle}
Let $C_n$ be the cycle on $n \ge 3$ vertices. Then $\mvr(\overline{C_n}) = 3$ if $n \ne 4$. 

In the case $n =4$, $\mvr(\overline{C_4}) =2$. 
\end{proposition}

We  note that, in the case $n \ne 4$, the orthogonal representation in \cite{ZFP} consists of pairwise linearly independent vectors, while the minimum dimension for an orthogonal representation for $\overline{C_4}$ is 2 but jumps up to 4 if we insist that all vectors be pairwise linearly independent. Such representations satisfy the hypotheses of Theorem 2.1 in \cite{OR}, allowing us to extend the methods of that paper relating to the complements of unicyclic graphs. 

\begin{proposition} \label{U comp.}
Let $U$ be a unicyclic graph, then
\bee
\mvr(\overline{U})=\left\{\begin{array}{ll}
4 & \textrm{if $\overline{L_4}$ is an induced subgraph of } U\\
2 & \textrm{if $U = C_4$}\\
3 & \textrm{otherwise}
\end{array}\right.
\eee
\end{proposition}

The graph $\overline{L_4}$ is the complement of a 2-tree as described in \cite{OR} and is shown in Figure \ref{L4Comp}.  Although this proposition can be seen as a special case of Proposition \ref{mvr leq 3}, the direct proof is an immediate application of the methods in \cite{OR}, so we include it here: 

\begin{figure}
\centering
\includegraphics[width = 2 cm, height = 3cm]{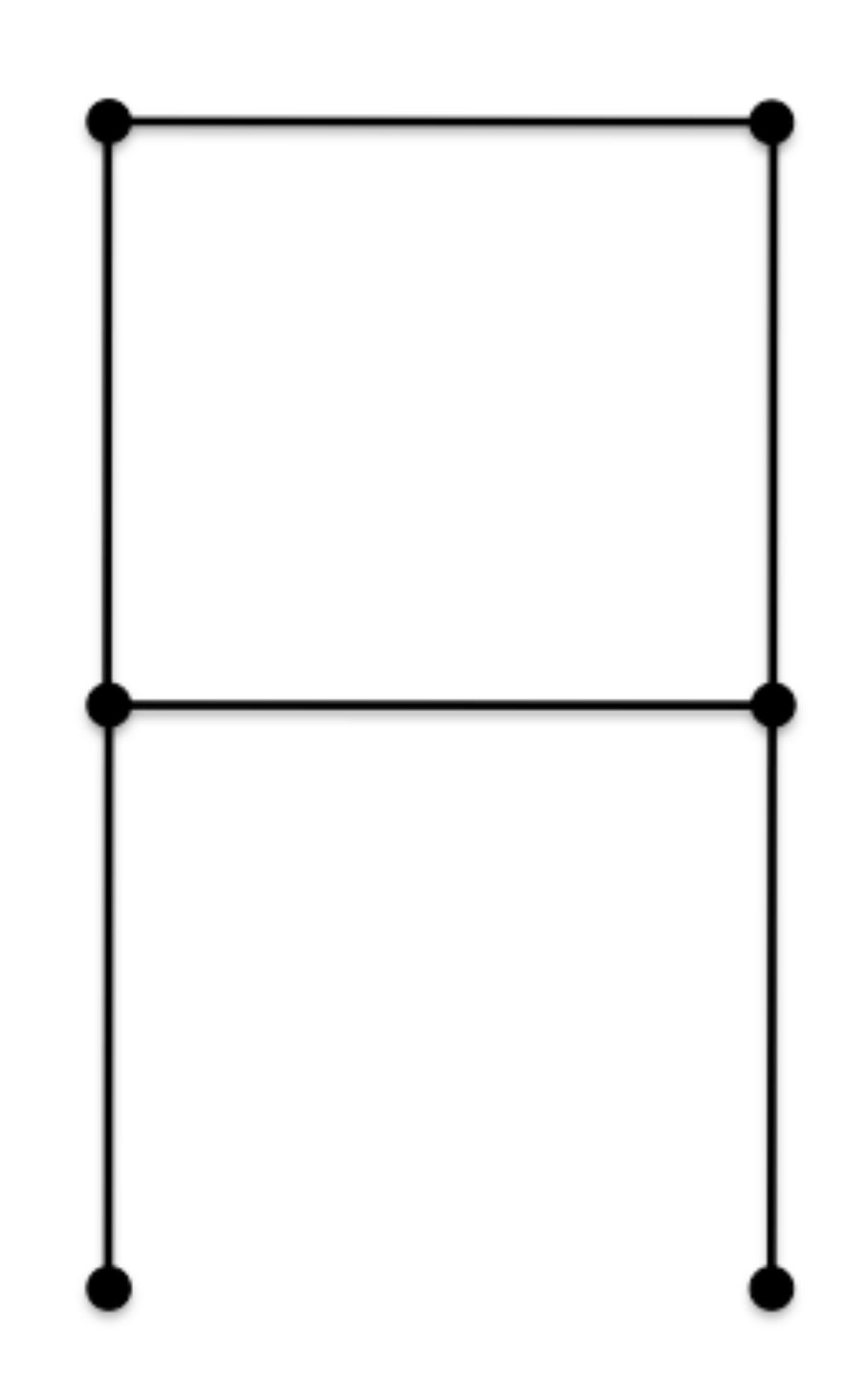}
\caption{The unicyclic graph $\overline{L_4}$, whose complement has minimum vector rank 4.} \label{L4Comp} \end{figure}
\begin{proof} Suppose that $U$ is built on a cycle of length $n \ne 4$. It follows from Proposition \ref{mvr comp. cycle} that the complement of this cycle has an orthogonal representation in $\reals^3$ in which any two vectors are linearly independent. Using the method of Corollary 3.4 from \cite{OR}, a unicyclic graph $U$ can be constructed from the cycle $C_n$ by adding one vertex at a time, with the new vertex adjacent to at most one prior vertex in $U$. The orthogonal representation of $U$ can similarly be built up one vector at a time. Thus, $\mvr(\overline{U})=\mvr(\overline{C_n})=3$ given that $n \ne 4$. 

If $n =4$, then there are two possibilities. If $U$ contains $\overline{L_4}$ as an induced subgraph, then $\mvr(\overline{U}) \ge \mvr(L_4) = 4$. Since $\mvr(\overline{U}) \le 4$ from \cite{OR}, we see that $\mvr(\overline{U})=4$ if $U$ contains $\overline{L_4}$ as an induced subgraph. 

If a unicyclic graph contains a 4-cycle but does not have $\overline{L_4}$ as an induced subgraph, then it contains a pair of diagonally opposite vertices of degree 2, as seen in Figure \ref{unicyclic4cycle}. These are duplicate vertices in $\overline{U}$. Removing one of these vertices leaves a tree, which means that $\mvr(\overline{U-v}) = 3$ unless $(U-v)$ is a star, in which case $U = C_4$. We can then reinstate the vertex $v$ and assign it the same vector as its duplicate. Thus, any unicyclic graph that does not have $\overline{L_4}$ as an induced subgraph will have $\mvr(\overline{U})= 3$ or $2$, depending on whether $U=C_4$, which is what we wished to show. \end{proof}
\begin{figure}[h]
\centering
\includegraphics[width=2cm,height=2cm]{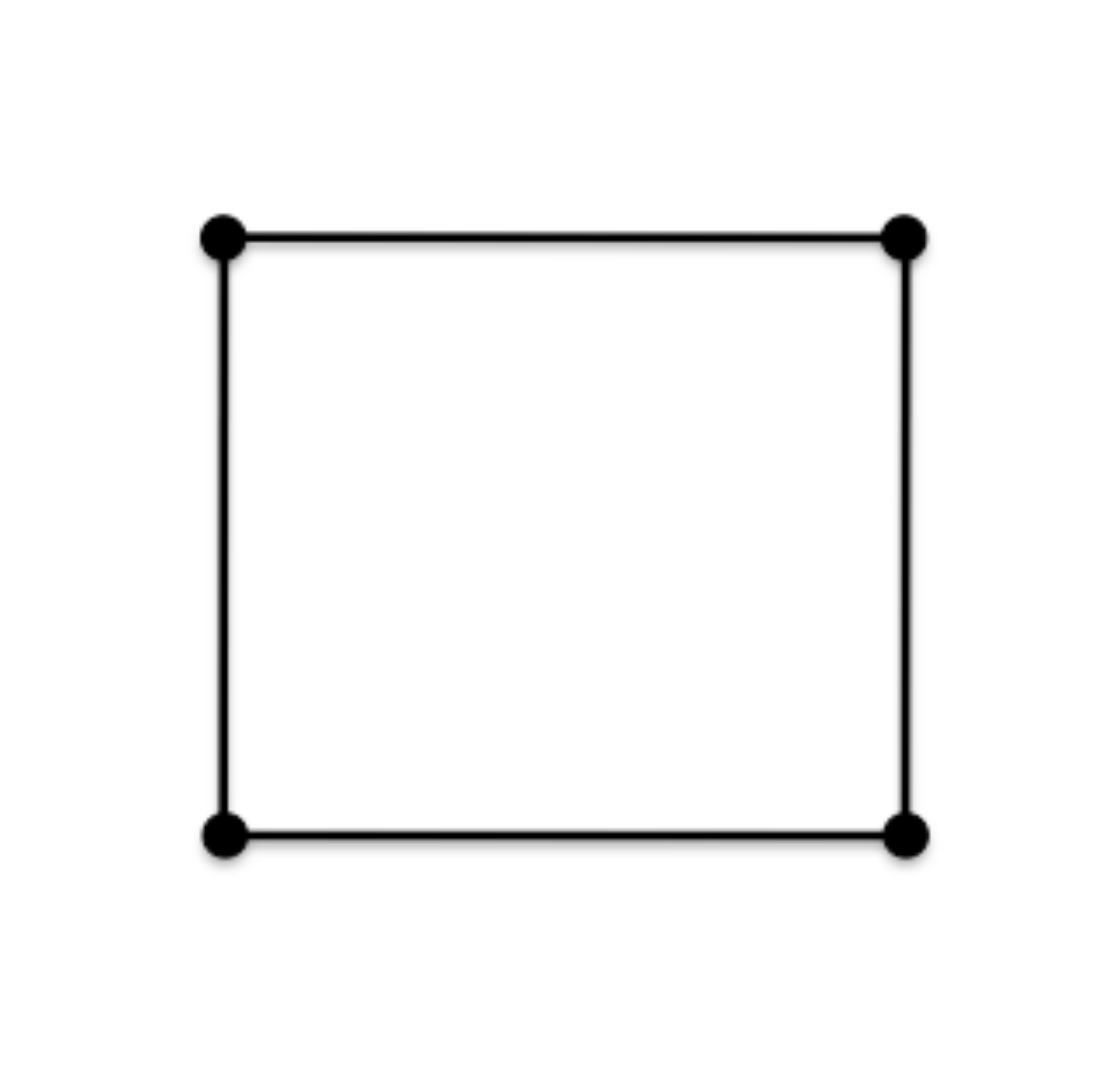} \hfil
\includegraphics[width=3.5cm,height=3cm]{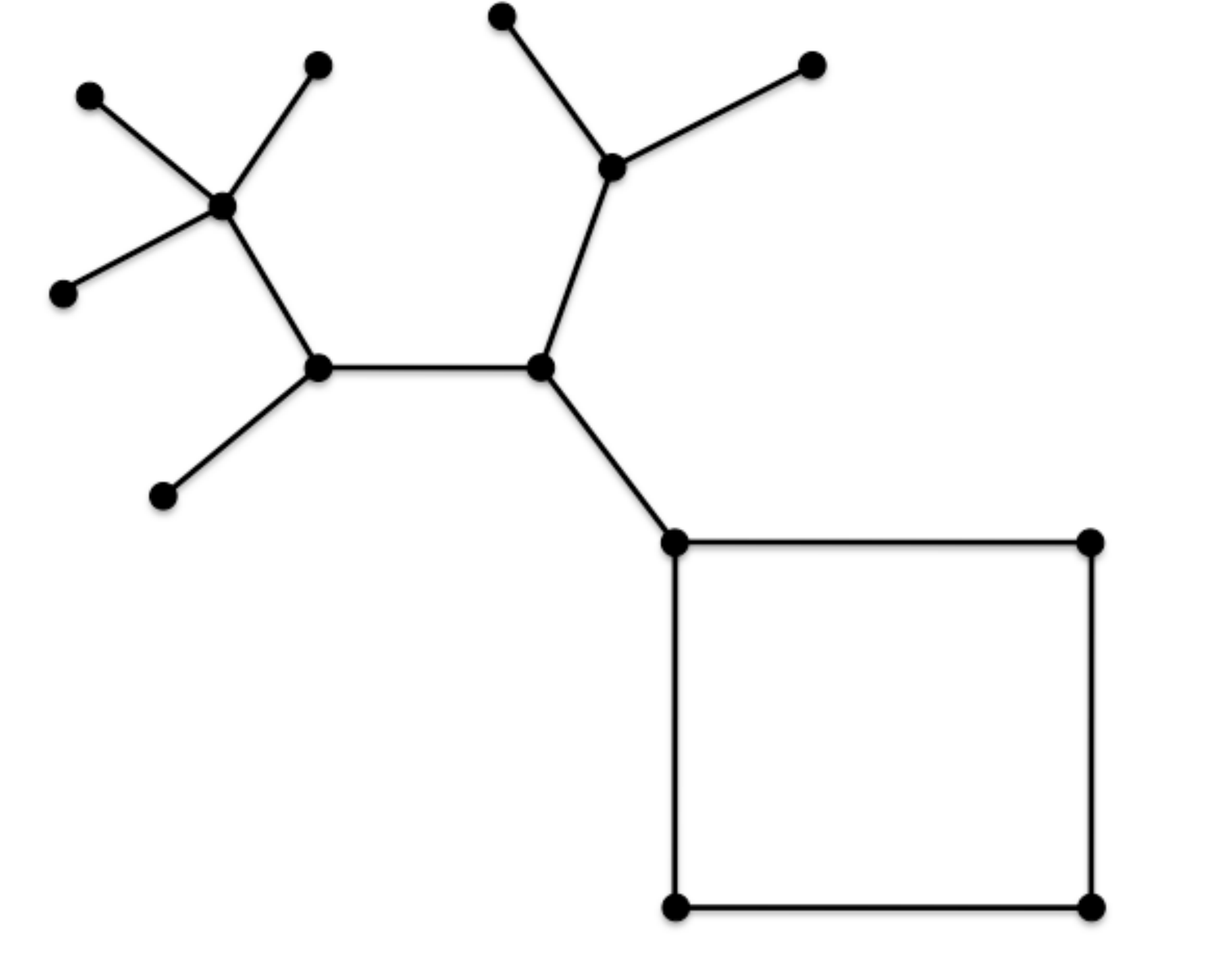} \hfil \includegraphics[width=3.5cm,height=3cm]{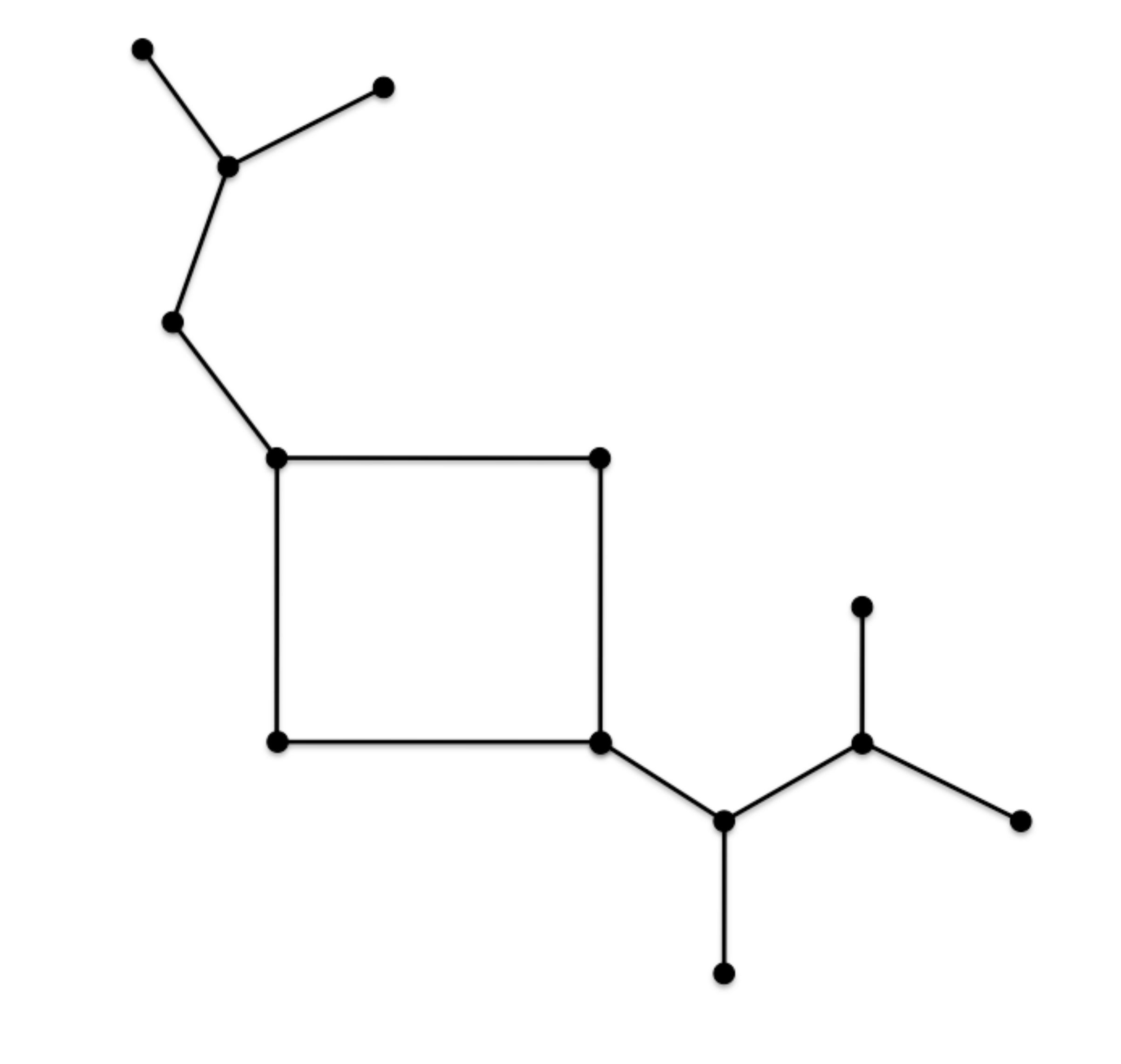}
\caption{Examples of unicyclic graphs with 4-cycles but no induced copies of $\overline{L_4}$}\label{unicyclic4cycle}
\end{figure}

\subsection{Adding long ears to graphs}\label{Section:Ears}

We would like to extend these ideas to calculate the minimum rank of the complement of graphs with more than one cycle. The following proposition will be a useful tool for this and is of interest in its own right. Recall that an {\em ear} of a graph is a maximal path subgraph whose internal vertices each have degree 2 in $G$. \cite{West} 

\begin{proposition}[Ear Decomposition]\label{ears} Let $G$ be a connected graph, and let $P$ be an ear of $G$ containing at least five vertices. Let $H$ be the subgraph of $G$ induced by removing the internal vertices of $P$. 

If $\overline{H}$ has an orthogonal representation in $\reals^3$ consisting of pairwise linearly independent vectors, then so does $\overline{G}$. 
\end{proposition}

\begin{proof} Suppose our ear consists of the path $u, w_1, w_2, \ldots w_{n-2}, v$, where $u$ and $v$ are vertices in $H$.  We can sequentially build the ear by adding vertices along the path, working from both endpoints, extending our three-dimensional vector representation at every step. Let $V'$ be the set of vertices in $H$, and $H = G[V']$ be the induced subgraph on those vertices. Using Theorem 2.1 of \cite{OR} twice, we  find vectors ${\bf w}_1$ and $\bf{w}_{n-2}$ in $\reals^3$ to give a vector representation of $G[V' \cup \{w_1, w_{n-2}\}]$ since $w_1$ and $w_{n-2}$ are pendant vertices in this graph. We can similarly add vectors for $w_2, w_3, \ldots, w_{n-4}$ in sequence, since these are added as pendant vertices. 

When we reach the final step, this method is no longer sufficient, as $w_{n-3}$ will have two neighbors in $G$ and two non-neighbors in $\overline{G}$. To find a vector representation that completes $\overline{G}$, we need Lemma \ref{prop:degree2}, which allows for tweaking the vectors assigned to the neighbors of $w_{n-3}$. This lemma is stated below and proved in the Appendix. The requirement that $n \ge 5$ ensures that when the two paths meet in the middle, that vertex is sufficiently isolated from the rest of the graph that we can adjust our representation to include it; and Lemma \ref{prop:degree2} shows us how to do this. This completes the proof that $\mvr(\overline{G}) = \mvr(\overline{G-w_{n-3}}) = 3$. \end{proof}

\begin{lemma}\label{prop:degree2}
Let $H$ be a graph with vertices $v$ and $w$ such that 
\begin{itemize}
\item{} $\deg(v) = \deg(w) =1$
\item{} $v$ and $w$ have no neighbors in common. 
\end{itemize}
Let $G$ be the graph formed from $H$ by adding a new vertex $u$ of degree 2 which is adjacent to $v$ and $w$. 

If there is an orthogonal representation of $\overline{H}$ in $\reals^3$ which assigns distinct nonzero vectors to each vertex, then  there is such an orthogonal representation of $\overline{G}$ in $\reals^3$. Hence, $\mvr(\overline{G}) \le 3$. 
\end{lemma}

We provide another simple result in the style of $\cite{OR}$ to cover ears of length 3. This will allow us to attach 3-cycles onto pendant vertices in the next section. The lemma is also proved in the Appendix. 

\begin{lemma}\label{3cycleslemma}
Let $H$ be a connected graph with pendant vertex $w$, and let $G$ be the graph formed by adding two new vertices $u$ and $v$ which are adjacent to $w$ and to each other but have no other neighbors. 

If $\overline{H}$ has an orthogonal representation in $\reals^3$ consisting of pairwise linearly independent vectors, then so does $\overline{G}$. 
\end{lemma}

\subsection{Applications of Proposition \ref{ears}}
Our motivation for Proposition \ref{ears} was to calculate the minimum vector rank of graphs defined in terms of cycles. The first family of graphs we look at are those in which no vertex lies on more than one cycle. These graphs are called {\em necklaces} in Section \ref{NecklaceSection} and include trees, cyclic, and unicyclic graphs; a generic example is shown in Figure \ref{N}. They can be built up step by step by either adding a pendant vertex adjacent to an existing vertex or building a cycle by adding ears between an existing pair of vertices. Hence, we can use Proposition \ref{ears} to establish a complement formula.
 
\begin{proposition}\label{mvr leq 3}
Let $N$ be a connected graph on at least two vertices such that no vertex lies on more than one cycle. Then \bee
\mvr(\overline{N})=\left\{\begin{array}{ll}
4 & \textrm{if $\overline{L_4}$ is an induced subgraph}\\
2 & \textrm{if $N =C_4$ or $N = K_{1,n}$ for } n\ge 1\\
3 & \textrm{otherwise}
\end{array}\right.\eee
\end{proposition}

\begin{proof}
The proposition has already been shown to be true for trees and unicyclic graphs. 
The proof then proceeds by induction on the number of cycles. Let us assume that the formula works for necklaces with fewer than $c \ge 2$ cycles. 

Let $N$ be a necklace with $c$ cycles which has no induced 4-cycles, and let $u$ be a cut-vertex such that $(N-u) = N' \cup F$ is the union of a necklace $N'$ on $c-1$ cycles and a forest $F$. (Such a cut-vertex must exist; its cycle is the analogue of a leaf in this ``tree of cycles.'') By assumption, $\mvr(\overline{N'}) = 3$. If $u$ is on a 3-cycle, we can apply Lemma \ref{3cycleslemma} to extend the representation of $\overline{N'}$ to the entire cycle. If $u$ is on a larger cycle, we can start with the representation of $\overline{N'}$ and use Theorem 2.1 from \cite{OR} to extend the representation and show that $\mvr(\overline{N'+u+v}) = 3$. We can then reconstruct the cycle by drawing a long ear from $v$ back to $u$ and use Proposition \ref{ears} to show that the minimum vector rank is not increased. Finally, we can reinstate the rest of the forest $F$, one pendant vertex at a time, without increasing the minimum vector rank until we have reconstructed our original graph $N$. Thus, $\mvr(\overline{N}) = 3$ for all necklace graphs which are 4-cycle-free. 

Now, suppose that $N$ is a necklace on $c\ge 2$ cycles which contains 4-cycles but no induced copies of $\overline{L_4}$. Repeating the argument from the proof of Proposition \ref{U comp.}, every 4-cycle contains a pair of diagonally opposite vertices of degree 2. Removing one vertex from each such pair yields a necklace $N'$ with no 4-cycles and hence $\mvr(\overline{N'}) \le 3$; and since $N$ has more than one cycle, $N'$ cannot be a star, which implying that $\mvr(\overline{N'}) \ge 3$. We then reinstate the duplicate vertices by duplicating its vector assignment and finding that $\mvr(\overline{N}) = 3$.

Finally, we consider the case where $N$ has an induced $\overline{L_4}$ subgraph. This immediately implies that $\mvr(\overline{N})  \ge 4$. We will show that $\mvr(\overline{N}) \le 4$ by building with a 4-cycle-free necklace. For each 4-cycle in $N$, split one of its edges into two by introducing a new vertex in the middle: $uv \rightarrow \{uw,wv\}$. This turns every four-cycle into a five-cycle and creates a new necklace $N'$ which is 4-cycle-free and which has $\mvr(\overline{N'}) = 3$. Now, embed this three-dimensional vector representation of $\overline{N}$ in $\reals^4$ by setting the fourth coordinate equal to zero. In order to make this a valid  vector representation of the original necklace $N$, for each 4-cycle $C_i$, we need to adjust the vectors assigned to our vertices $u_i$ and $v_i$ by assigning them respective fourth coordinates $x_i$ and $y_i$ so that $x_iy_i = -\bra {\bf u}_i, {\bf v}_i\ket \ne 0$. Since we have infinitely many independent choices for each pair $(x_i,y_i)$, we can find a labeling such that  $\bra {\bf u}_i,  {\bf v}_j \ket \ne 0$ for $i \ne j$. This guarantees that a representation can be found in $\reals^4$, which implies that $\mvr(\overline{N}) = 4$. \end{proof}

A different generalization of unicyclic graphs are those in which all induced cycles share a single common edge. These are called books in Section \ref{BookSection}. 

\begin{defn}\rm A {\em book} is a connected graph which has a unique edge $e$ such that the intersection of any two induced cycles of $B$ is exactly the edge $e$ and its vertices. 
\end{defn}
This definition is discussed more in Section \ref{BookSection}. Note that the uniqueness of $e$ implies that $B$ contains at least two cycles. 
\begin{proposition}\label{mvr comp. B-v}
For a book $B$ as defined above, 
\bee
\mvr(\overline{B})=\left\{\begin{array}{ll}
4 & \textrm{if either $C_4$ or the kite $\kappa$ is an induced subgraph of } B\\
3 & \textrm{otherwise}
\end{array}\right.
\eee
\end{proposition}
\begin{proof}

Let $B$ be a book with at least two cycles which does not have any induced kite or 4-cycles. Start with a single 3-cycle (if there is one) or else any other induced unicyclic graph. From here, we can build $B$ up one cycle at a time by adding long ears using Propostion \ref{ears} to show that  $\mvr(\overline{B})=3$. If $B$ has multiple 3-cycles, then each contains the two binding vertices and a third vertex of degree 2 in $B$ (since $B$ is kite-free). In this case,  the set of ``third vertices'' in $\overline{B}$ consists of all duplicate vertices, and we can assign the same vector to each of them. 

The kite on five vertices is the only book whose complement is a connected tree. As such, its minimum vector rank is $|B|-1 = 4$, implying that $\mvr(\overline{B}) \ge 4$. Likewise, if $B$ has more than one cycle and contains $C_4$ as an induced subgraph, it implies that $B$ must have an induced subgraph isomorphic to at least one of the 4-cycle graphs shown in Figure \ref{BookComps}, each of which has complement minimum vector rank of 4. Since each of these four graphs has an orthogonal representation in $\reals^4$ which satisfies the conditions of Theorem 2.2 of \cite{OR}, we can build up the rest of $B$ one vertex at a time and maintain $\mvr(\overline{B})\leq 4$. This means that if $B$ contains either a kite or a four-cycle as a subgraph, then $\mvr(\overline{B}) = 4$.\end{proof}

\begin{figure}
\centering
\includegraphics[width=2cm,height=3cm]{L4_comp} \hfil \includegraphics[width=2cm,height=3cm]{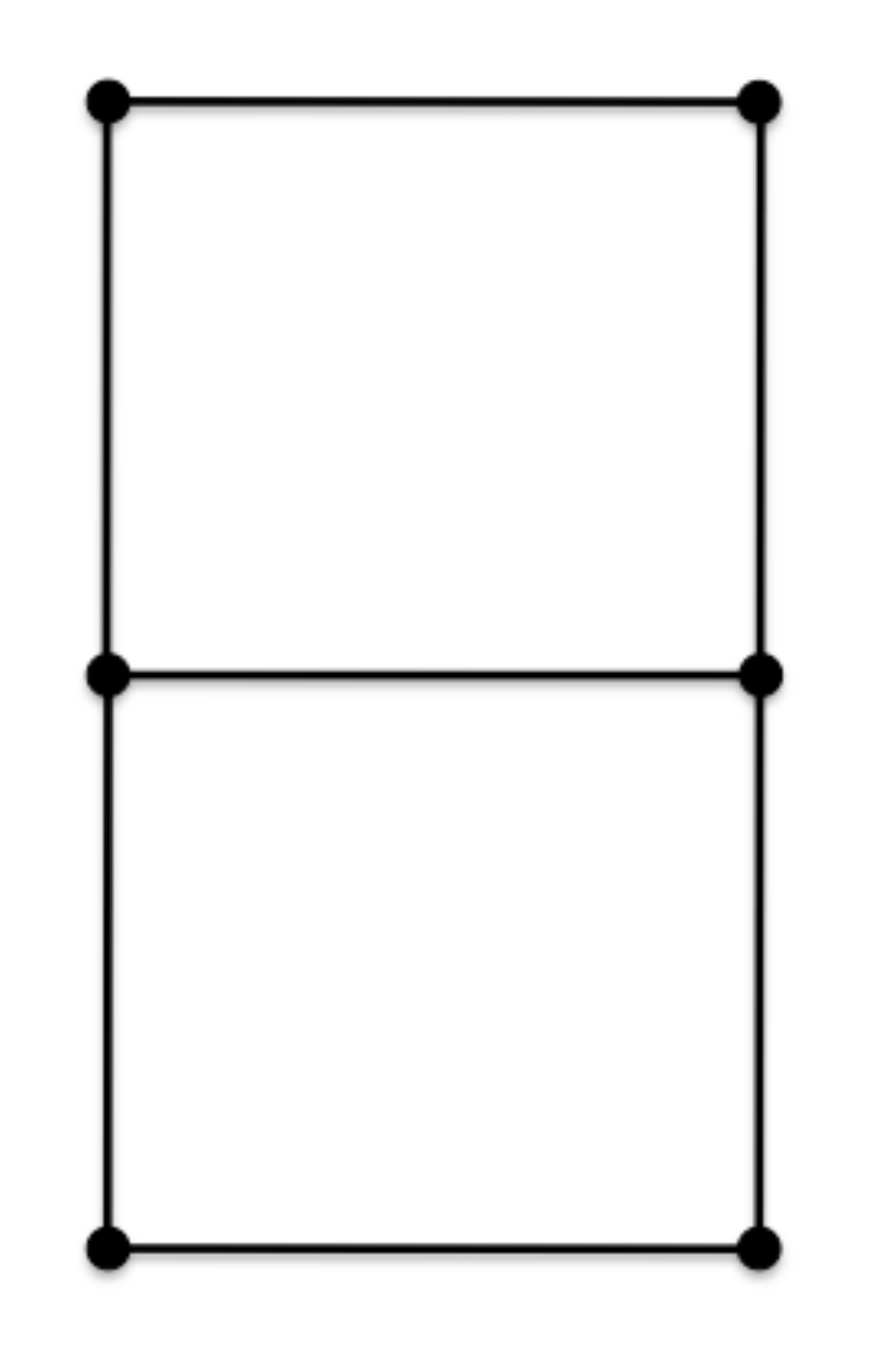} \hfil 
\includegraphics[width=2cm,height=2.5cm]{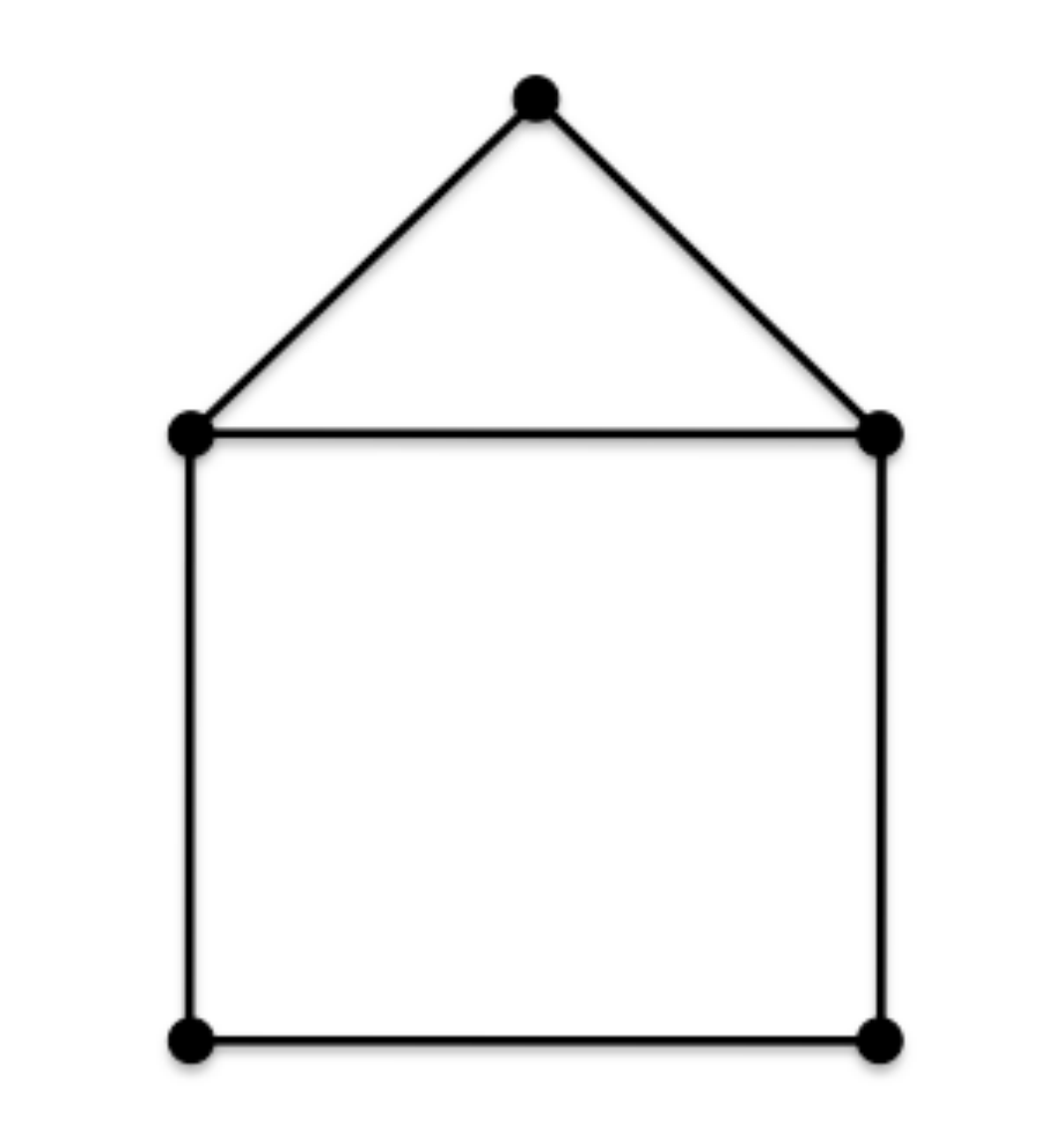}\hfil \includegraphics[width=3cm,height=3cm]{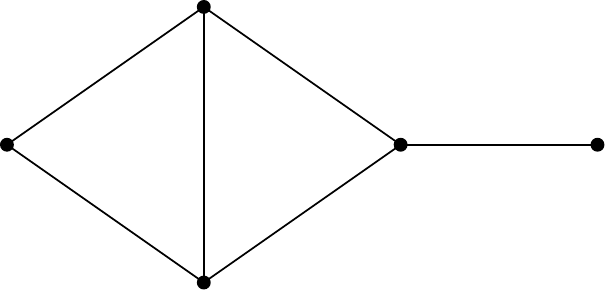} 
\caption{Books for which $\mvr(\overline{B}) = 4$: $\overline{L_4}$, Domino, $\overline{P_5}$, Kite $\kappa$}\label{BookComps}
\end{figure}


\section{Families of Complement Critical Graphs}\label{Families}
In an arbitrary graph $G$, it is not obvious in general whether it is complement critical. However, if the graph has nice properties,we can sometimes formulate simple criteria. This section focuses on identifying families of vector critical and complement critical graphs. The initial section shows that complement critical graphs which are not connected must be built from smaller complement critical graphs. We then derive necessary and sufficient conditions for certain types of graphs to be complement critical based on their induced cycles. 

Although the definition of complement critical graphs uses a property of {\em all} induced subgraphs of $G$, it is sufficient to consider only those induced by removing a single vertex of $G$, since $G$ is vector-critical if and only if $\mvr(G-v) < \mvr(G)$ for all vertices $v$. Likewise for complement criticality. We will make use of this fact repeatedly in what follows, writing $(G-v)$ to indicate the induced subgraph of $G$ on all of its vertices except $v$. 

\subsection{Disconnected graphs}
The first result shows that disconnected complement critical graphs are composed of complement critical components. This will allow us to focus the remainder of our work on connected graphs.
\begin{proposition}\label{G joint}
Let $G$ be a graph which is not connected. We can write $G=G_1 \cup {G_2}$ with $\mvr(\overline{G_1}) \le \mvr(\overline{G_2})$. Then $G$ is vector critical if and only if both $G_1$ and $G_2$ are vector critical. 

In addition, $G$ is complement critical if and only if the following three conditions are all satisfied:
\begin{itemize}
\item $\mvr(\overline{G_1})<\mvr(\overline{G_2})$
\item $G_1$ is vector critical.
\item $G_2$ is complement critical.
\end{itemize}
\end{proposition}
Note that there is a possible ambiguity in the labeling of $G_1$ and $G_2$ if  $\mvr(\overline{G_1}) = \mvr(\overline{G_2})$; but in this case, $G$ cannot be complement critical. 
\begin{proof} We proceed by removing a single vertex $v$ and looking at the induced subgraph.  This vertex may be removed from $G_1$ or from $G_2$. 

In the case where $v$ is removed from $G_1$, we note that  $\mvr(\overline{G_1-v}) \le \mvr(\overline{G_1}) \le \mvr(\overline{G_2})$, which implies that $\mvr(\overline{G-v}) = \mvr(\overline{G_2})=\mvr(\overline{G})$. This means that removing a vertex from $G_1$ will never reduce the minimum vector rank of $\overline{G}$. Likewise, we can see that 
\bee \mvr(G) - \mvr(G-v) = \mvr(G_2) = \mvr(G_1) -  \mvr(G_1-v) \eee

Therefore, $\mvr(G-v) + \mvr(\overline{G-v}) < \mvr(G) + \mvr(\overline{G})$ if and only if $\mvr(G_1) >\mvr(G_1-v)$. This condition holds  for all vertices $v$ of $G_1$ if and only if $G_1$ is critical. 

In the case where $v$ is removed from $G_2$, we can calculate that  $\mvr(\overline{G-v}) = \max(\mvr(\overline{G_1}),\mvr(\overline{G_2-v}))$.  This means that $\mvr(\overline{G-v}) = \mvr(\overline{G})$ if and only if either $\mvr(\overline{G_1}) = \mvr(\overline{G_2})$ or $\mvr(\overline{G_2}) = \mvr(\overline{G_2-v})$. 

So, in order for $G$ to be complement critical, either $\mvr(G_2 -v) < \mvr(G_2)$ for all $v$ (i.e., $G_2$ is critical), or else we need $G_2$ to be complement critical {\em and} have $\mvr(\overline{G_1}) < \mvr(\overline{G_2})$. \end{proof}

By repeated application of this result, we can see that every connected component of a complement critical graph must be complement critical; and, in fact, all but one of them must be vector critical as well. This allows us to focus our remaining results on graphs for which both $G$ and $\overline{G}$ are connected. 

\subsection{Trees}\label{trees}
Trees are the simplest connected graphs, and it is straightforward to characterize them in terms of criticality:

\begin{proposition}\label{Tree, Critical}
Any tree is complement critical.  Any tree which is not a star is vector-critical.
\end{proposition}

\begin{proof} If $T$ is a tree on $n$ vertices, then $\mvr(T)=n-1$; if we remove a single vertex $v$, then $\mvr(T-v) \le |T-v| -1 = n-2$ unless $(T-v)$ contains only isolated vertices. In this case, $T$ must be a star and $v$ must be the central vertex. 

Conclusion: If $T$ is not a star, then for any vertex $v$ $\mvr(T-v) \le n-2 < \mvr(T)$, so $T$ is vector-critical. 

If $T = K_{1,n-1}$ and the vertex being removed is the dominating vertex, then $\mvr(T-v) = \mvr(T)$ but $\mvr(\overline{T-v}) = \mvr(K_{n-1}) = 1$. Since $\mvr(\overline{T}) = 2$ by Proposition \ref{mvr tree comp.}, we see that a star is complement critical. \end{proof}

\subsection{Unicyclic Graphs}

Following the discussion of trees, which contain no cycles, it is natural to look at graphs which contain only one. For a connected unicyclic graph $U$, it was shown in \cite{Outerplanar} that $\mvr(U) = |U| - 2$. We use this result and Proposition \ref{U comp.} to characterize connected unicyclic graphs which are  complement critical.
\begin{proposition}\label{U Comp. Critical}
A connected unicyclic graph $U$ is complement critical if and only if at least one of the following three conditions is true:
\begin{enumerate}
\item{} Each vertex on the cycle of $U$ is adjacent to at least three vertices of degree greater than 1. In this case, $U$ is vector-critical. 
\item  $\overline{L_4}$ is an induced subgraph of $U$. 
\item  $U = S_n^3$, which is the graph formed from the star graph $K_{1,n-1}$ by connecting two of the leaves with an edge, as shown in Figure \ref{Sn3}. 
\end{enumerate}
\end{proposition}

\begin{figure}[h]
\centering
\includegraphics[width=4cm,height=3cm]{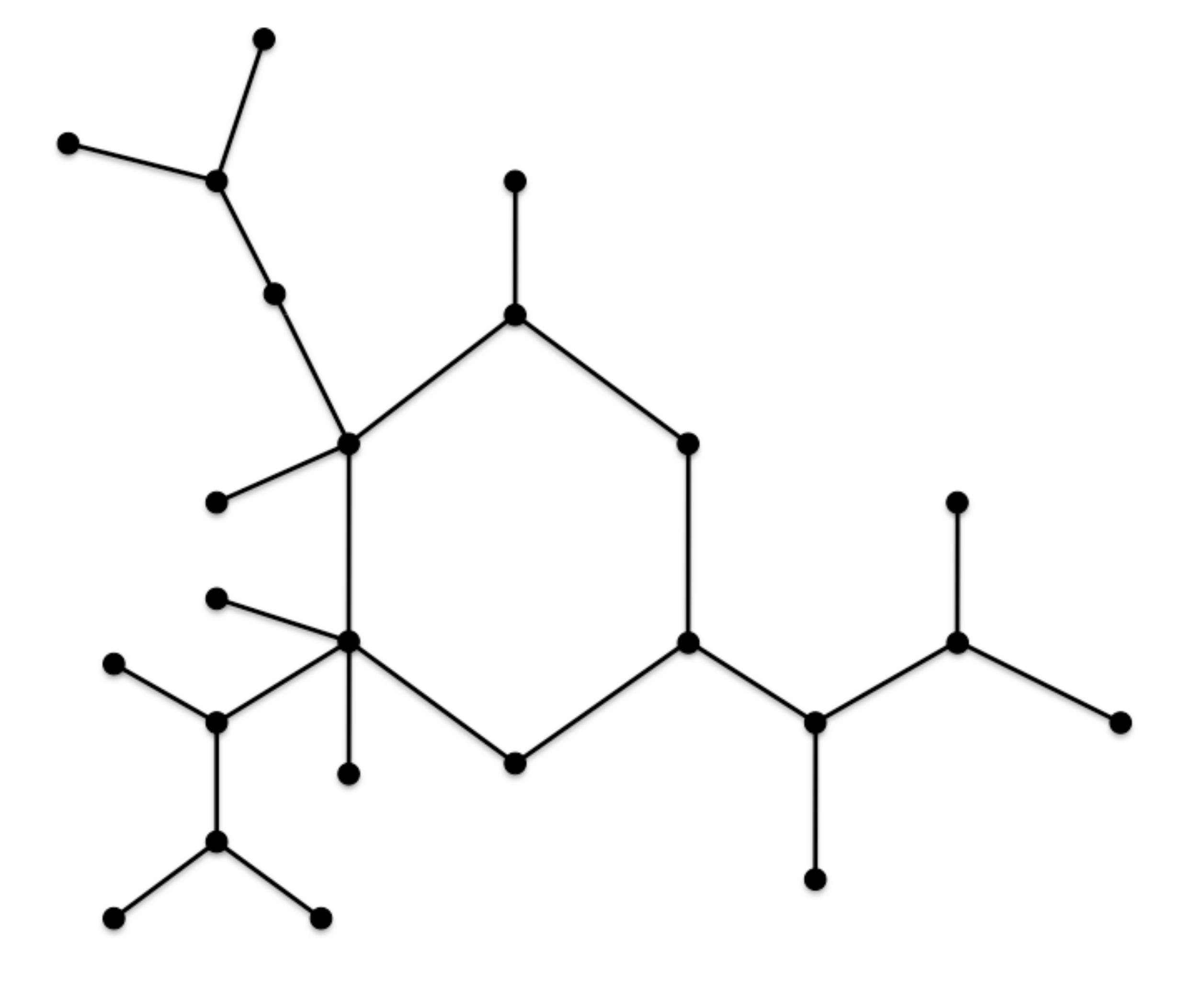} \qquad 
\includegraphics[width=2cm,height=2cm]{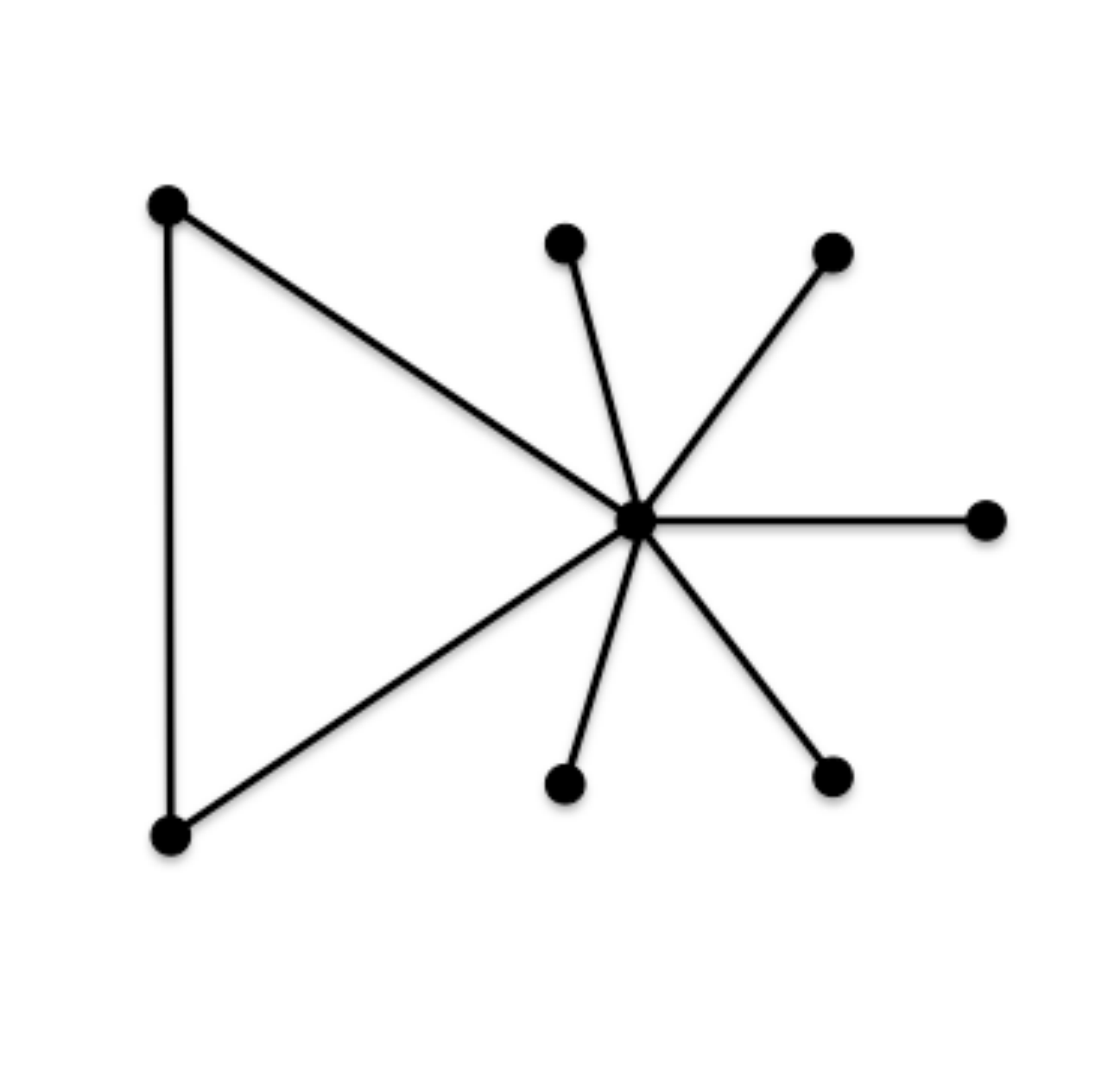}
\caption{Generic unicyclic graph (left) and the graph $S_8^3$ (right) }
\label{Sn3}
\end{figure}

\begin{proof}
We first observe that if $v$ is not on the cycle of $U$, then $(U-v)$ is the union of a connected unicyclic graph $U'$ and a (possibly empty) forest $F$. This implies that 
\bee
\mvr(U-v) = \mvr(U') + \mvr(F)  \le  |U'| - 2 + |F| = |U| - 3 
\eee
Thus, removing a vertex not on the cycle will always reduce $\mvr(U)$.  This reduces the question of $U$'s complement criticality to what happens when you remove a vertex on the cycle. To do this, we can check directly that each of the three conditions implies that $U$ is complement critical:

Condition (1) implies that for every $v$ on the cycle, $(U-v)$ is a forest containing at least two proper trees. This means that $\mvr(U-v) \le  |U-v| - 2  <  \mvr(U)$. 

Condition (2) implies that $\mvr(\overline{U}) = 4$. If $v$ is on the cycle, then $(U-v)$ is a forest and $\mvr(\overline{U-v}) \le 3 < \mvr(\overline{U})$. 

Finally, (3) implies that if $v$ is on the cycle, then either $(U-v) = K_{1,n-2}$ or $(U-v) = \overline{K_{n-1}}$. In either case, $\mvr(\overline{U-v}) \le 2 < \mvr(\overline{U})$. 

We conclude that any one of these three conditions implies that $U$ is complement critical. Now, we wish to show that they are in fact necessary.  Assume that $U$ is a connected unicyclic graph which is complement critical for which the conditions (1) and (2) do not hold. We wish to show that $U = S_n^3$ for some $n$.


Since condition (1) does not hold, there is at least one vertex $v$ on the cycle for which $(U-v)$ is the union of a tree $T$ plus a collection of $r \ge 0$ isolated vertices.  This means that $\mvr(U-v) = |T| - 1 + r = |U| - 2 = \mvr(U)$. By assumption, $U$ is complement critical, so $\mvr(\overline{U-v}) < \mvr(\overline{U})$. Condition (2) does not hold, which means that $\mvr(\overline{U})\leq 3$ by Proposition \ref{U comp.}. This forces $\mvr(\overline{U-v}) \le 2$. Proposition \ref{mvr tree comp.} then tells us that $T$ must be a star.

Let $u$ be the center vertex of $T$, and let $w \notin \{u,v\}$ be any vertex from the cycle in $U$. $w$ is necessarily a leaf in $T$, which means that  $\deg(w) \le 2$ in $U$. This means that $w$ does not satisfy condition (1), so we can repeat our argument (interchanging the roles of $v$ and $w$) to show that $(U-w)$ must also include a star with $v$ as a leaf and that $\deg(v)  \le 2$. Since $v$ and $w$ are degree two vertices which are adjacent to each other and to $u$, we see that $U$ must be built on a 3-cycle and that only one of these three vertices can have degree greater than 2, leading us to conclude that $U = S_n^3$ for some $n \ge 3$. 

Conclusion: If $U$ is complement critical and neither condition (1) nor (2) holds, then condition (3) must be true. This gives us the biconditional in the proposition. \end{proof}

As an immediate consequence, we see that a cyclic graph $C_n$ is complement critical if and only if $n = 3$. The three-cycle $C_3 =S_3^3$ is a degenerate form of Figure \ref{Sn3}; its complement $\overline{C_3} = 3K_1$ is vector critical.

\subsection{Necklaces}\label{NecklaceSection}
The next type of graph we consider is a generalization of trees and unicylic graphs. We define a {\em necklace} to be a connected graph with in which each vertex belongs to at most one cycle, as shown in Figure \ref{N}. A proper necklace must have at least two vertices. Graphs in which each {\em edge} appears in at most one cycle are called cactus graphs and are well-studied \cite{West}, but there does not appear to be as much attention to graph in which no two cycles intersect even on the vertices. Necklaces can be thought of as a ``tree of cycles'', since we can form a tree by collapsing the vertices of each cycle into a single vertex. Such graphs are planar and, since each cycle is actually an induced cycle, there are no internal vertices and necklace graphs are outerplanar. The infinite face of this graph touches every vertex. This will allow us to use results from \cite{Outerplanar}. 

\begin{figure}[h]
\centering
\includegraphics[width=7cm,height=5.5cm]{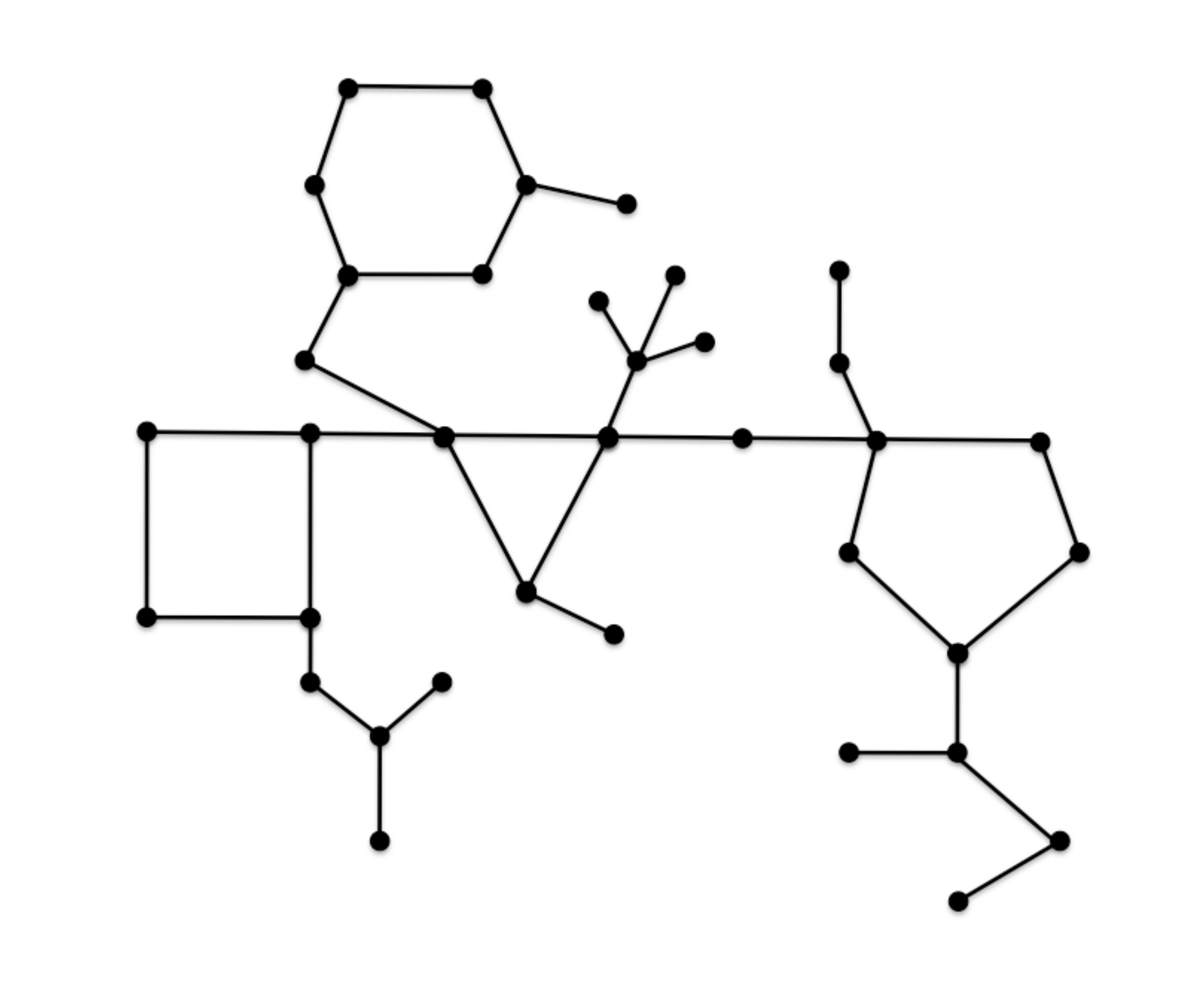}
\caption{A necklace graph with four cycles}
\label{N}
\end{figure}

In order to characterize complement critical necklaces, it will be useful to have formulas for the minimum rank. The minimum rank of the complement was described in Proposition \ref{mvr leq 3}, while the minimum rank of a necklace itself generalizes the formula for trees and unicyclic graphs:
\begin{proposition}\label{necklaceformulas}
If $N$ is a proper necklace, then $\mvr(N) =|N| - c -1$. 
\end{proposition}
Note that the formula specifies to unicyclic graphs and proper trees (with $c=1$ and $c=0$, respectively) but does not apply if $N = K_1$. The simplest proof of this formula uses the result from \cite{Outerplanar}, which states that since $N$ is outerplanar and connected, $\mvr(N) = \mr_+(N) = |N| - T(N)$, where $T(N)$ is the minimum number of vertex-disjoint induced trees of $N$ needed to cover all of its vertices. If we choose one vertex from each cycle and delete its two edges on the cycle, what remains is a forest consisting of $c+1$ induced trees. This is a tree cover, hence $T(N) \le c+1$. On the other hand, in any a tree covering of $N$, each cycle must be covered using at least two induced trees. As a result, the covering must omit at least two edges from each cycle (where the induced trees connect to each other). Since $N$ has $|N| - 1 +c$ edges, the tree covering has at most $|N| -1 - c$ edges. A union of $k$ trees covering $|N|$ vertices has exactly $|N| -k$ edges, which implies that any tree covering of $N$ uses at least $c+1$ trees.  Thus $T(N) = 1+c$ and $\mvr(N) = |N| - c- 1$. \cvd.

The following proposition gives the main result of this section.
\begin{proposition}\label{necklacecc}
A necklace $N$ is vector critical if and only  every vertex on a cycle of $N$ is adjacent to at least three vertices of degree greater than 1.

A necklace $N$ is complement critical if and only if one of the following three conditions is true:
\begin{enumerate}
\item $N$ is vector critical. 
\item $N$ has $\overline{L_4}$ as an induced subgraph; every induced copy of $\overline{L_4}$ is built on the same 4-cycle; 
and every vertex which lies on a cycle other than this 4-cycle is adjacent to at least three vertices of degree greater than 1.
\item $N=S_n^3$ (See Figure \ref{Sn3}.)
\end{enumerate}
\end{proposition}

\begin{proof}
This is a generalization of Proposition \ref{U Comp. Critical}, although the conclusions are not as simple. As in the unicyclic case,  if $v$ is a vertex which is not on a cycle, then $(N-v) = \bigcup_{i = 1}^k N_i \cup \overline{K_r} $ is the union of $k \ge 1$ proper necklaces plus $r \ge $ isolated vertices.  The total number of cycles is this same as in the original graph. This implies that $\mvr(N-v) = r + \sum_i \mvr(N_i) = r + \sum_i (|N_i| - c_i - 1) = |N| - 1 - c - k < \mvr(N)$ since $k \ge 1$.  

Likewise, if $v$ is on a cycle and has at least three neighbors of degree greater than 1, then $(N-v)$ is a union of $k>1$ proper necklaces plus $r \ge 0$ isolated vertices; and the total number of cycles is now one less. Thus, $\mvr(N-v) = r + \sum_i \mvr(N_i) =   |N|-1 - (c-1) - k <|N| - c- k < \mvr(N)$. 

On the other hand, if $v$ is on a cycle and has only two neighbors of degree greater than 1, then $k = 1$ in the above calculation and we get $\mvr(N-v) = |N|-1 - (c-1) - k = \mvr(N)$. This implies that if $v$ is on a cycle, then $\mvr(N-v) < \mvr(N)$ if and only if $v$ has at least three vertices of degree greater than 1.  This implies that $N$ is vector critical if and only if every vertex that lies on a cycle satisfies the neighbor condition. 

When we look at the complement $\overline{N}$, we see that the proof for unicyclic graphs shows both ways in which $\mvr(\overline{N-v}) < \mvr(\overline{N})$:
\begin{itemize}
\item{} $\mvr(\overline{N}) = 4$ and $\mvr(\overline{N-v}) < 4$, which happens when $\overline{L_4}$ is an induced subgraph of $N$ but not $(N-v)$. 
\item{} $\mvr(\overline{N}) = 3$ and $\mvr(\overline{N-v}) < 3$, which happens when $(N-v)$ is a star. \end{itemize}

This explains why $N$ can only have one copy of $\overline{L_4}$ and limits the possibilities for having  $\mvr(\overline{N-v}) < \mvr(\overline{N})$ to those listed in the proposition. \end{proof}

\subsection{Books}\label{BookSection}
The last type of graph we explore is a book, which is a generalization of the books described in \cite{ZFP}.  A {\bf book} is defined to be a graph with a distinguished edge $e$ such that the intersection of any two induced cycles in $B$ is the edge $e$ and its endpoints. This generalizes the definition in \cite{ZFP}, in which it was assumed that all the cycles were the same size and that every edge of $B$ was on a cycle. Figure \ref{B} shows a book with a 3-cycle, a 5-cycle, and two 4-cycles; it is apparent that such graphs are always planar. Books are outerplanar only when there are at most two induced cycles, but like outerplanar graphs, the chromatic number is at most 3. 

\begin{figure}
\centering
\includegraphics[width=6cm,height=5cm]{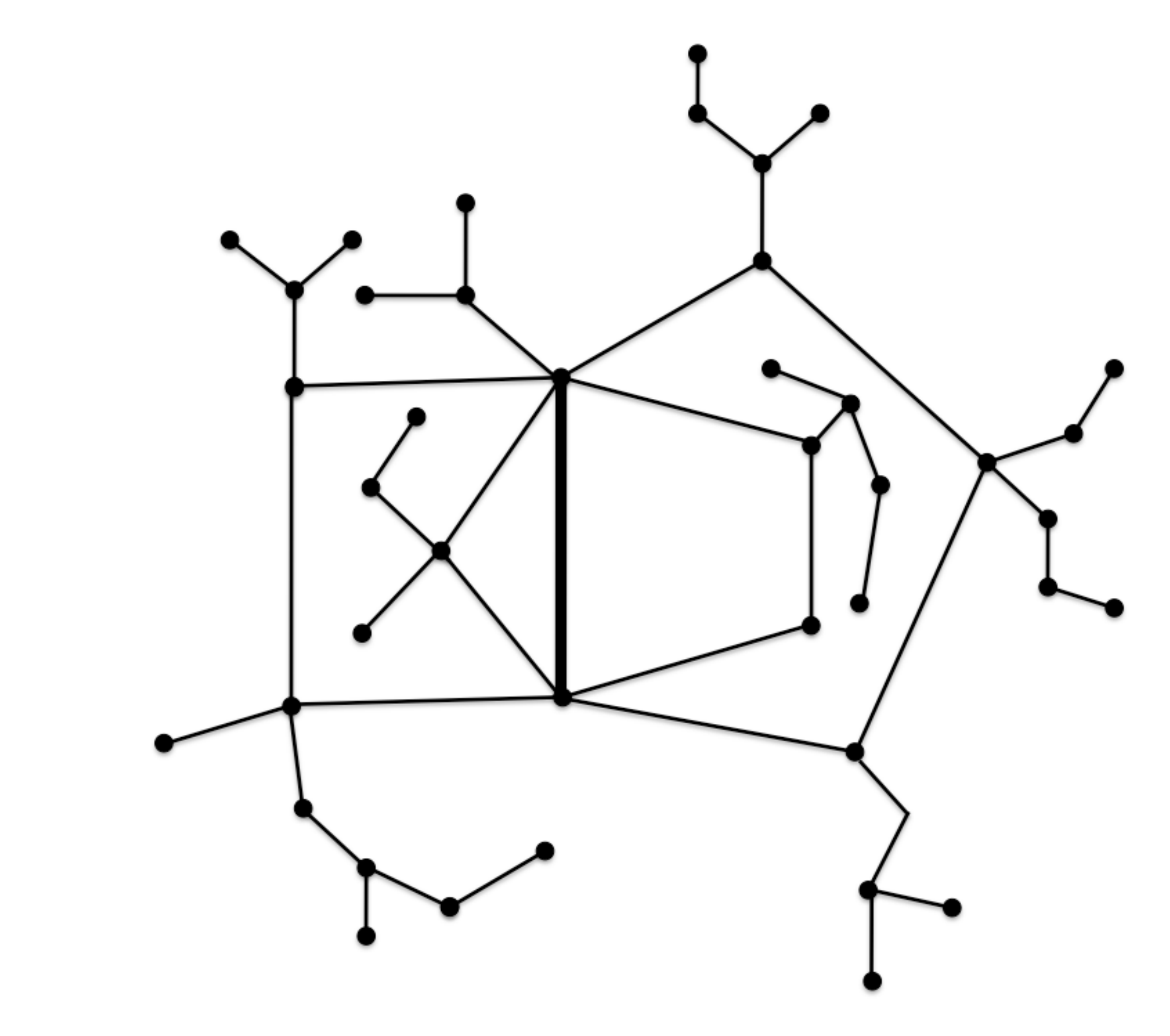}
\caption{A book with four induced cycles}
\label{B}
\end{figure}
We will assume in what follows that $B$ is connected and that it contains more than one cycle. This uniquely defines the edge $e$ and its vertices, which are called {\bf binding vertices} and denoted as $v_b$. This is a different direction in which to generalize unicyclic graphs; but we derive results similar to those in previous sections. 

First, we observe that the proof of Proposition 3.11 in \cite{ZFP} extends to our graphs, as the binding vertices form a positive forcing set for the any book $B$, which implies that $\mvr(B) \ge |B| - 2$. Since $B$ is not a tree, this is also an upper bound and $\mvr(B) = |B|-2$. A description of the minimum vector rank of the complement of a book were given in Proposition \ref{mvr comp. B-v}. Armed with these results, we can give necessary and sufficient conditions for a book to be complement critical: 
\begin{proposition}\label{comp. critical books}
A book $B$ with at least two cycles is complement critical if and only if it satisfies one of the following conditions:
\begin{enumerate}
\item Each binding vertex has at least one neighbor which does not lie on a cycle and which has degree greater than 1. In this case, $B$ is vector-critical. 
\item Either the four-cycle $C_4$ or the kite $\kappa$ is an induced subgraph of $B$. 
\item Every cycle of $B$ is a 3-cycle and every vertex of $B$ is adjacent to at least one of the binding vertices. 
\end{enumerate}
\end{proposition}

\begin{proof}  

One simplifying observation if that if $v$ is not one of the binding vertices, then $\mvr(B-v) < \mvr(B)$. This is because the induced subgraph $(B-v)$ has a connected component $B'$ which is still a book. We write $(B-v) = B' \cup F$, where $F$ is the (possibly empty) remainder of the graph (which is a forest) and which has the trivial bound $\mvr(F) \le |F|$. Therefore we have $\mvr(B-v)=\mvr(B')+\mvr(F)  \leq |B'|-2+|F|=|B|-3<\mvr(B)$. 

This means that the complement criticality of a book depends solely on its binding vertices. The removal of a binding vertex $v_b$ will give a forest with $t\ge1$ nontrivial trees and some isolated vertices. This means that $\mvr(B - v_b) = |B|-1 - t < \mvr(B)$ if and only if $t > 1$. Thus, $B$ will be vector-critical if and only if each binding vertex has a neighbor which lies off the cycle and which has degree at least 2. 

Otherwise, we need to ask whether $\mvr(\overline{B - v_b}) < \mvr(\overline{B})$. Since $(B-v_b)$ is a forest, $\mvr(\overline{B-v_b}) \le 3$; thus if $\mvr(\overline{B}) = 4$, we automatically have complement criticality. Using Proposition \ref{mvr comp. B-v}, this gives the second condition. 

The only other possibility would be for $\mvr(\overline{B-v_b}) =2$. This will happen only if each tree in $(B-v_b)$ is a star. In particular, the other binding vertex must be adjacent to every other vertex in its component, which means that each cycle must be a triangle with no extra vertices attached. Also, the remaining binding vertex cannot have any other neighbors of degree greater than 1. This specifically precludes the first condition, which means that if we need $\mvr(\overline{B-v_b}) =2$ for one of the binding vertices, we need it for the other. This inevitably leads us to the 3-cycle graph $B^3_m$ ($m \ge 2$), possibly with leaves attached to the binding vertices, as shown in Figure \ref{Bm3}. \end{proof}
\begin{figure}[hf]
\centering
\includegraphics[width=2.5cm,height=2.5cm]{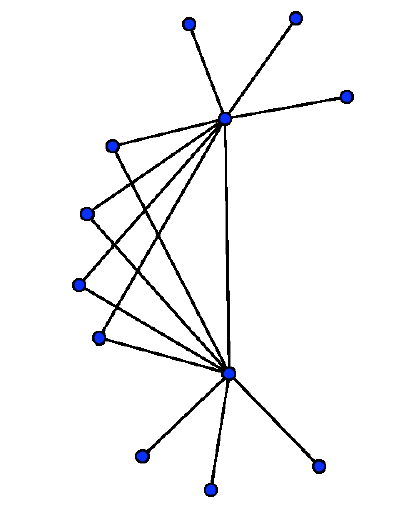}
\caption{$B^3_m$ is a set of $m$ $3$-cycles built on a common edge. We can then attach an arbitrary number of pendant vertices to each binding vertex. The graph shown is built from $B^3_4$.}
\label{Bm3}
\end{figure}
Because books can be treated as gluing multiple unicyclic graphs together by one common edge, it is not surprising that their properties resemble those of unicyclic graphs. Returning to the examples from \cite{ZFP}, we get the following corollary:
\begin{corollary}
The graphs $B^t_m$ described in \cite{ZFP} are complement critical if and only if $t= 3$ or $t = 4$.  
\end{corollary}

These graphs consist of $m$ copies of a $t$-cycle glued together along a single edge. There are no vertices which do not lie on a cycle. Looking at the Proposition, we see that the graph $B^3_m$ clearly meets the third condition while $B^4_m$ meets the second. For any other value of $t$, $B^t_m$ is complement critical if and only if it meets the first condition, which it cannot since there are no vertices away from the cycles. Thus, no $B^t_m$ if complement critical if $t>4$. \cvd

\section{Complement Critical Graphs and the Minimum Semidefinite Rank}\label{MRplus}

The notion of a critical graph must always be defined in terms of some graph function; in the previous section, all calculations were done with respect to the minimum vector rank. This brief section is meant to point out that all of our results can be readily adapted using the minimum semidefinite rank ($mr_+$), which is more widely used in the literature. The only significant difference between the two measures is on a single isolated vertex, where $\mvr(K_1) = 1$ but $\mr_+(K_1) = 0$; and the inequality $\mvr(G) \ge \mr_+(G)$ always holds, with equality guaranteed for connected graphs. 

This gives an immediate indication that being complement critical with respect to $
\mvr$ is a stronger one than being complement critical with respect to $\mr_+$: 
\begin{proposition}
Let $G$ be any connected graph such that $\overline{G}$ is also connected. If $G$ is vector-critical with respect to the minimum vector rank, then it is vector-critical with respect to the minimum semidefinite rank.   If $G$ is complement critical with respect to the minimum vector rank, then it is complement critical with respect to the minimum semidefinite rank.  
\end{proposition}

The proof is immediate. If $G$ is connected and $\mvr(G-v) < \mvr(G)$, then $\mr_+(G-v) \le \mvr(G-v) < \mvr(G) = \mr_+(G)$; and the same argument can be made for $\overline{G}$. 

In terms of the results for families of complement critical graphs given in Section \ref{Families}, we need no longer be concerned with the appearance of isolated vertices in induced subgraphs of $G$. This simplifies some of the conditions: 

\begin{proposition} Let $G$ be a connected graph which is not complement critical with respect to the minimum vector rank. 
\begin{enumerate}
\item{} If $G$ is a unicyclic graph or a necklace and each vertex on its cycle has degree at least three, then $G$ is vector-critical with respect to the minimum semidefinite rank. 
\item{} If $G$ is a necklace graph which has $\overline{L_4}$ as an induced subgraph such that every induced copy of $\overline{L_4}$ is built on the same 4-cycle; 
and every vertex which lies on a cycle other than this 4-cycle has degree at least 3,  then $G$ is complement critical with respect to the minimum semidefinite rank.
\item{} If $G$ is a book in which each binding vertex has at least one neighbor which does not lie on a cycle, then $G$ is critical with respect to the minimum semidefinite rank.  
\end{enumerate}
\end{proposition}

Examples of graphs which are critical and complement-critical with respect to the minimum semidefinite rank but not the minimum vector rank are shown in Figure \ref{msrExamples}. The simplest of these is the $n$-sun (discussed in \cite{AS} \cite{Outerplanar}), which is the cycle $C_n$ with a pendant vertex attached to each vertex on the cycle. This graph is critical since for any vertex on the cycle, $\mr_+(G-v) = (|G| - 2)-1 < \mr_+(G)$. However, since the isolated vertex in $(G-v)$ has minimum vector rank equal to 1, $\mvr(G -v) = \mvr(G)$. The figure also shows examples of graphs which meet conditions (2) and (3) from the Proposition.

\begin{figure}[h] \includegraphics[height = 2.8 cm]{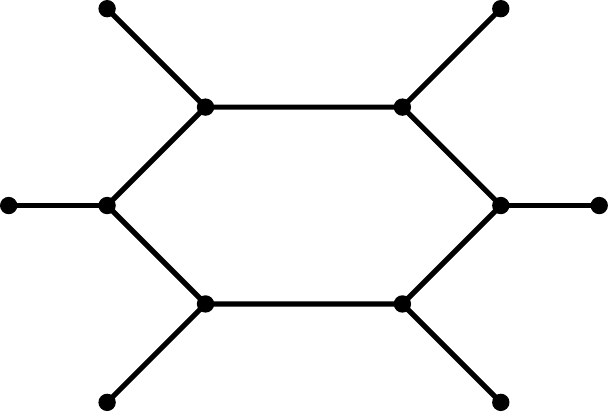}\hfil \includegraphics[height = 3 cm]{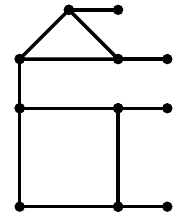} \hfil \includegraphics[height = 3 cm]{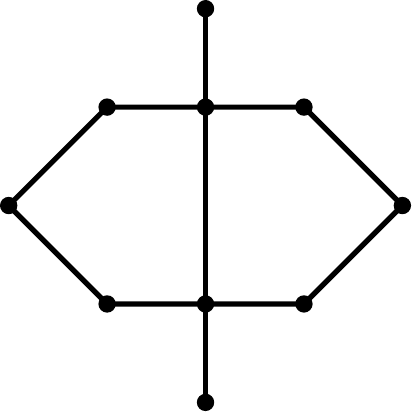}
\caption{Examples of graphs which are complement critical with respect to the minimum semidefinite rank but not the minimum vector rank}\label{msrExamples} 
\end{figure}

\section{Conclusion}

In this paper, we introduce the notion of  complement critical graphs defined with respect to the minimum vector rank. These objects occupy an important place in the set of graphs when partially ordered by induced subgraphs. We showed a link between complement critical graphs and the Graph Complement Conjectures, showing that these graphs are sufficient to prove or disprove the conjecture. We provided a list of general properties of complement graphs and necessary and sufficient conditions for graphs which are built on cycles to be complement critical. In the process, we gave explicit formulas for the complements of certain sparse graphs. We anticipate that this work will be useful in understanding the structure of the set of graphs given by the minimum rank. 


\bigskip
{\bf Acknowledgment.} This work was primarily completed during the Saint Mary's College of California School of Science Summer Research Program. The authors gratefully acknowledge the support of this program and its many benefits.


\appendix\section{Proofs of Necklace Properties}
\subsection{Proof of Proposition \ref{prop:degree2}}
Assume that there exists a vector labeling  of the vertices of $\overline{H}$ in $\reals^3$ that assigns pairwise linearly independent nonzero vectors to each vertex. Let ${\bf v}$ and ${\bf w}$ be the vectors corresponding to the vertices $v$ and $w$;  let ${\bf {\bf y_1}}$ and ${\bf y_2}$ be the vectors corresponding to the respective neighbors of $v$ and $w$; and let ${\bf y_3}, \ldots {\bf y_k}$ be the vectors corresponding to the remaining vertices of $H$.  This is represented schematically in Figure \ref{fig:degree2}

\begin{figure}[h]
\centering \includegraphics[width=4.5cm]{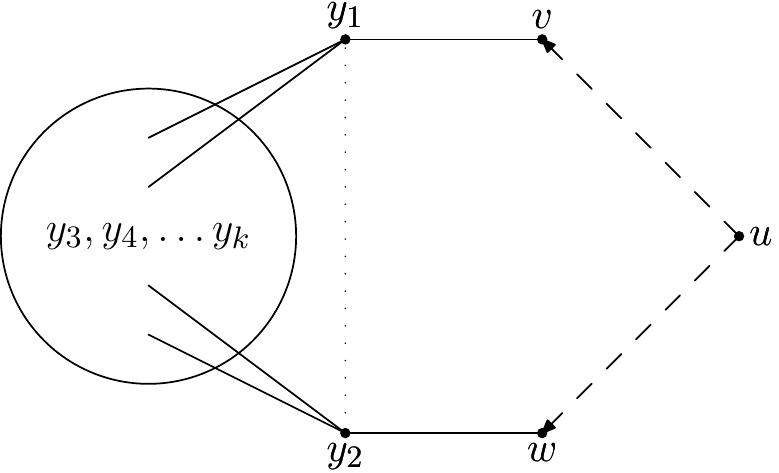}\caption{Schematic Representation of Graph in Lemma \ref{prop:degree2}}\label{fig:degree2}
\end{figure}

It will be convenient to take advantage of being in $\reals^3$ and use the cross product. For any real number $t$, define
\bee
{\bf u}_t := t {\bf v} +{\bf w} \qquad 
{\bf v}_t := {\bf u}_t \times {\bf y_1} \qquad
{\bf w}_t := {\bf u}_t \times {\bf y_2}
\eee

We wish to show that there exists a value of $t \in \reals$ for which $\{{\bf u}_t, {\bf v}_t, {\bf w}_t,  {\bf y_1}, \ldots {\bf y_k}\}$ is an orthogonal representation of $\overline{G}$. 

By construction, for all $t$, \bee \bra {\bf u}_t,{\bf v}_t \ket = \bra {\bf u}_t,{\bf w}_t \ket = \bra {\bf v}_t,{\bf y_1}\ket = \bra {\bf w}_t,{\bf y_2}\ket = 0 \eee
Since these are the only non-neighbors of our vertices $u,v,w$, we require only that no other inner product be zero. Specifically, we need to show that there exists a value of $t$ such that:
\begin{description}
\item{(I)} $\bra {\bf u}_t, {\bf y_i} \ket \ne 0$ for all $i = 1 \ldots k$
\item{(II)} $\bra {\bf v}_t, {\bf y_i} \ket \ne 0$ for $i \ne 1$ and  $\bra {\bf w}_t, {\bf y_i} \ket \ne 0$ for $i \ne 2$
\item{(III)} $\bra {\bf v}_t, {\bf w}_t \ket \ne 0$
\end{description}
These inner products are linear and quadratic functions of $t$. The result is proved as long as none of them is identically zero.

Looking at line (I), $\bra {\bf u}_t, {\bf y_i} \ket$ is identically zero if and only if $\bra {\bf v},{\bf y_i} \ket = \bra  {\bf w}, {\bf y_i} \ket = 0$, which would imply that $v$ and $w$ share a common neighbor in $H$.  In line (II), $\bra {\bf v}_t, {\bf y_i} \ket = t\det({\bf v}, {\bf y_1}, {\bf y_i})  +\det({\bf w}, {\bf y_1}, {\bf y_i}) = 0$ for all $t$ if and only if ${\bf y_i}$ is in both the plane spanned by  ${\bf v}$ and ${\bf y_1}$ {\em and} the plane spanned by ${\bf w}$ and ${\bf y_1}$. This means that ${\bf y_i}$ and ${\bf y_1}$ are dependent, which implies that $i = 1$. The second part of (II) is similarly shown. 

Finally, in (III), we can write the quadratic function $f(t) = \bra {\bf v}_t, {\bf w}_t \ket$. If $f''(0) = 0$, then $\bra {\bf y_1}, {\bf y_2}\ket = 0$ and $f'(0) = \bra {\bf v},{\bf y_2} \ket\bra {\bf w}, {\bf y_1}\ket \ne 0$. Hence $f''(0)$ and $f'(0)$ are not both zero and $f(t)$ cannot be identically zero. 

Since each condition excludes finitely many values of $t$, there exists a value of $t$ which makes $\{{\bf u}_t, {\bf v}_t, {\bf w}_t, {\bf y_1}, \ldots, {\bf y_k}\}$ an orthogonal representation of $\overline{G}$ in$ \reals^3$. 
\cvd

\subsection{Proof of Lemma \ref{3cycleslemma}}
Suppose that we have an orthogonal representation of $H$ in $\reals^3$ in which $w$ is represented by the unit vector ${\bf w}$. Observation 1.3 in \cite{OR} guarantees that there exists a unit vector ${\bf u_0}$ in $\reals^3$ which is orthogonal to ${\bf w}$ but not to any other vector in the representation. Let ${\bf v_0} = {\bf u_0} \times {\bf w}$. As in the previous proof, we define
\bee
{\bf u}_t = {\bf u_0} + t{\bf v_0}  \qquad 
{\bf v}_t = t{\bf u_0} -{\bf v_0} \eee
The vectors ${\bf u}_t, {\bf v}_t$ and ${\bf w}$ are mutually orthogonal for all $t$, and for any other vector ${\bf y}$ in the orthogonal representation, $\bra {\bf y}, {\bf u}_t \ket$ and $\bra {\bf y}, {\bf v}_t\ket$ are not identically zero because $\bra {\bf y}, {\bf u_0}\ket \ne 0$.  Therefore, we can find a value of $t$ for which ${\bf u}_t$ and ${\bf v}_t $ extend the orthogonal representation of $\overline{H}$ to $\overline{G}$. 

\subsection{Proof of Proposition \ref{mvr leq 3}}
Note that if $G$ is a disconnected graph with components $G_i$, then $\mvr(\overline{G}) = \max_i \mvr(\overline{G_i})$ and $\overline{L_4}$ is an induced subgraph of $G$ if and only if it is a subgraph of some $G_i$. So, if the proposition is true for connected graphs, then it must be true for all simple graphs.

We also note that one direction of the proof is clear: Theorem 2.2 of \cite{OR} implies that $\mvr(\overline{G}) \le 4$, and if $\overline{L_4}$ is an induced subgraph of $G$, then $\mvr(\overline{G}) \ge \mvr(L_4) =4$.  The proof of the converse comes from the fact that any such graph $G$ can be built up vertex by vertex in a way that uses the previous two results. 

We proceed by induction on the number of vertices. The proposition evidently holds for $K_2$. Assume now that it holds for all connected graphs with fewer than $n>2$ vertices in which every vertex belongs to at most one cycle. Now, consider such a graph $G$ with $n$ vertices such that $\overline{L_4}$ is not an induced subgraph. Note that every induced subgraph of $G$ also satisfies the hypotheses of the proposition.

If $\overline{G}$ contains a pair of duplicate vertices $u$ and $v$, then let $H = G-v$. Since duplicate vertices don't change the $\mvr$, $\mvr(\overline{G}) = \mvr(\overline{H}) \le 3$ by the inductive hypothesis. 

Now, assume that $\overline{G}$ contains no duplicate vertices. Since $G$ does not have $\overline{L_4}$ as an induced subgraph, this implies that $G$ contains no 4-cycle. 

If $G$ contains a pendant vertex $v$, then Theorem 2.1 in \cite{OR} asserts that $\mvr(\overline{G}) = \mvr(\overline{G-v})$ as long as any two vectors in the representation of $\overline{G-v}$ are linearly independent, which is equivalent to having no duplicate vertices. Thus, $\mvr(\overline{G}) \le 3$ by the inductive hypothesis. 

Assume now that $G$ contains no pendant vertices. Thus, $G$ is simply a bunch of cycles connected by paths. Note that we can associate a tree $T$ with $G$ by collapsing each cycle of $G$ into one vertex of $T$. Because a tree has at least two leaves, we see that $G$ has at least two cycles which are attached to the rest of the graph by a single edge. Let us consider one such cycle $C$ attached to $G$ at $w$. We know that it is not a 4-cycle, since this would include duplicate vertices. If $C$ is a 3-cycle on $(u,v,w)$, let $H = G-u-v$ and apply Lemma \ref{3cycleslemma}.  If $C$ is a cycle of length greater than four, let $u$ be a vertex on $C$ at distance 2 from $w$. Let $H = G-u$ and apply Proposition \ref{prop:degree2}. In either case, we see that $\mvr(\overline{G}) = \mvr(\overline{H}) \le 3$ by the inductive hypothesis. 

Thus, $\mvr(\overline{G})\le 3$ if and only if $\overline{L_4}$ is not an induced subgraph of $G$. \cvd


\begin{thebibliography}{1}

\bibitem{ZFS} AIM Minimum Rank--Special Graphs Work Group. Zero forcing sets and the minimum rank of graphs.  \textit{Linear Algebra and its Applications}, 428:1628--1648, 2008.

\bibitem{AIMW} American Institute of Mathematics workshop, ``Spectra of Families of Matrices described by Graphs, Digraphs, and Sign Patterns,'' held Oct. 23--27, 2006 in Palo Alto, CA. Materials available at \url{http://aimath.org/WWN/matrixspectrum/matrixspectrum.pdf}









\bibitem{GCC} F. Barioli, W. Barrett, S. M. Fallat, H. T. Hall, L. Hogben, H. Holst. On the graph complement conjecture for minimum rank.  \textit{Linear Algebra and its Applications}, 436: 4373--4391, 2012. 

\bibitem{ZFP} F. Barioli, W. Barrett, S. M. Fallat, H. T. Hall, L. Hogben, B. Shader, P. van den Driessche, H. van der Holst. Zero forcing parameters and minimum rank problems.  \textit{Linear Algebra and its Applications}, 433 2: 401--411, 2010.

\bibitem{Outerplanar} F. Barioli, S. M. Fallat, L.H. Mitchell, S.K. Narayan. Minimum Semidefinite Rank of Outerplanar Graphs and the Tree Cover Number.  \textit{Electronic Journal of Linear Algebra}, 22: 10-21, 2011.

\bibitem{MSR} M. Booth, P. Hackney, B. Harris, C. R. Johnson, M. Lay, T. Lenker, L. H. Mitchell, S. K. Narayan, A. Pascoe,  B. D. Sutton. On the minimum semidefinite rank of a simple graph.  \textit{Linear and Multilinear Algebra}, 59: 483--506, 2011. 


\bibitem{OMR} M. Booth, P. Hackney, B. Harris, C. R. Johnson, M. Lay, L. H. Mitchell, S. K. Narayan, A. Pascoe, K. Steinmetz, B. D. Sutton, W. Wang. On the minimum rank among positive semidefinite matrices with a given graph.  \textit{SIAM Journal of Matrix Analysis and Applications}, 30: 731--740, 2008.

\bibitem{QCN} P.J. Cameron, M.W. Newman, A. Montanaro, S. Severini, A. Winter. On the quantum chromatic number of a graph. \textit{Electronic Journal of Combinatorics}, 14, 2007.

\bibitem{DSW} R. Duan, S. Severini, A. Winter. Zero-error communication via quantum channels and a quantum Lov{\'a}sz $\vartheta$-function. In \textit{2011 IEEE Int. Symposium on Info. Th. Proc.}, 2011.

\bibitem{AS} S. Fallat, L. Hogben. The minimum rank of symmetric matrices described by a graph: a survey.  \textit{Linear Algebra and its Applications}, 426: 558--582, 2007. 

\bibitem{SII}  S. Fallat, L. Hogben. Variants on the minimum rank problem: a survey II.  \url{http://arxiv.org/pdf/1102.5142.pdf}, 2010.

\bibitem{OVC} G. Haynes, C. Park, A. Schaeffer, J. Webster, L. H. Mitchell. Orthogonal vector coloring.  \textit{Electronic Journal of Combinatorics}, 17, 2010.

\bibitem{OR} L. Hogben. Orthogonal representations, minimum rank, and graph complements.  \textit{Linear Algebra and its Applications}, 428: 2560--2568, 2008.

\bibitem{UM} Y. Jiang, L.H. Mitchell, S.K. Narayan. Unitary matrix digraphs and minimum semidefinite rank. \textit{Linear Algebra and its Applications}, 428:1685--1695, 2008.




\bibitem{Nylen} P. M. Nylen. Minimum-Rank Matrices with Prescribed Graph.  \textit{Linear Algebra and its Applications}, 248: 303--316, 1996. 


\bibitem{West} D.B. West. {\em Introduction to Graph Theory}. Prentice Hall Inc., Upper Saddle River, NJ, 1996. 
\end{thebibliography}
\end{document}